\documentclass[11pt, oneside]{article}   	% use "amsart" instead of "article" for AMSLaTeX format
\usepackage{geometry}                		% See geometry.pdf to learn the layout options. There are lots.
\usepackage{graphicx}				% Use pdf, png, jpg, or eps§ with pdflatex; use eps in DVI mode
\usepackage[all]{xy}
\CompileMatrices

\usepackage{amssymb}
\usepackage{amsmath}
\usepackage{stmaryrd}
\usepackage{amsthm}
\usepackage{tikz-cd}
\usepackage{extarrows}
\usepackage[mathscr]{euscript}
\usepackage{mathrsfs} 
\usepackage{verbatim}
\usepackage[OT2,T1]{fontenc}
\DeclareSymbolFont{cyrletters}{OT2}{wncyr}{m}{n}
\DeclareMathSymbol{\Sha}{\mathalpha}{cyrletters}{"58}
\DeclareMathSymbol{\Che}{\mathalpha}{cyrletters}{"51}

\usepackage{calligra}
\usepackage{mathrsfs}
\newcommand{\calHom}{\mathscr{H}\mathit{om}}
\newcommand{\calMor}{\mathscr{M}\mathit{or}}
\newcommand{\calEnd}{\mathscr{E}\mathit{nd}}

\newcommand{\Ga}{{\mathbf{G}}_{\rm{a}}}
\newcommand{\Gm}{{\mathbf{G}}_{\rm{m}}}

\DeclareMathOperator{\perf}{perf}

\DeclareMathOperator{\Pic}{Pic}
\DeclareMathOperator{\Jac}{Jac}

\DeclareMathOperator{\Hom}{Hom}
\DeclareMathOperator{\End}{End}
\DeclareMathOperator{\Mor}{Mor}

\DeclareMathOperator{\R}{R}

\DeclareMathOperator{\im}{im}

\newcommand*{\Z}{\ensuremath{\mathbf{Z}}}                        % integers
\newcommand*{\A}{\ensuremath{\mathbf{A}}}                        % affine/adele
\renewcommand*{\P}{\ensuremath{\mathbf{P}}}                        % proj space
\newcommand*{\calO}{\mathcal{O}}                                  % 'sheaf' O
\newcommand*{\address}{Einstein Institute of Mathematics, The Hebrew University of Jerusalem, Edmond J. Safra Campus, 91904, Jerusalem, Israel}
\newcommand*{\email}{zev.rosengarten@mail.huji.ac.il}

\usepackage{rotating}

\usepackage{bm}

\numberwithin{equation}{section}

\newtheorem{theorem}{Theorem}[section]

\newtheorem{lemma}[theorem]{Lemma}
\newtheorem{proposition}[theorem]{Proposition}
\newtheorem{corollary}[theorem]{Corollary}

\theoremstyle{definition}
  \newtheorem{definition}[theorem]{Definition}
  
  \theoremstyle{remark}
  \newtheorem{example}[theorem]{Example}

\theoremstyle{remark}
  \newtheorem{remark}[theorem]{Remark}

\usepackage[OT2,T1]{fontenc}

\tikzset{commutative diagrams/.cd,
mysymbol/.style={start anchor=center,end anchor=center,draw=none}
}

\usepackage{stmaryrd}
\usepackage{hyperref}

\title{\textbf{MODULI SPACES OF MORPHISMS INTO SOLVABLE ALGEBRAIC GROUPS}}
\author{Zev Rosengarten \thanks{While completing this work, the author was supported by a Zuckerman Postdoctoral Scholarship. \newline
}}

\date{}
\begin{document}
\maketitle

\begin{abstract}
We construct (in significant generality) moduli spaces representing (on the category of geometrically reduced schemes over the base field) the functor of morphisms from a scheme into a solvable algebraic group.
\end{abstract}

\tableofcontents{}

%%% Section: Introduction
\section{Introduction}

Moduli spaces form a fundamental tool and object of study in modern and classical algebraic geometry. One important example is the moduli space of morphisms between a pair of schemes. Such moduli spaces often exist when the source scheme is projective; see, for example, \cite[\S5.6, Th.\,5.23]{fgaexplained}. In general, however, for morphisms between affine schemes one cannot expect such representability results. The purpose of the present paper is to nevertheless construct such schemes of morphisms in significant generality when the target is a solvable algebraic group over a field.

To set the stage, let us recall that a smooth connected unipotent group $U$ is said to be {\em $k$-split} (or just split when $k$ is clear from the context) if it admits a $k$-group scheme filtration {\em over $k$}
\[
1 = U_0 \trianglelefteq U_1 \trianglelefteq \dots \trianglelefteq U_n = U
\]
with $U_{i+1}/U_i \simeq \Ga$ for $0 \leq i < n$. As opposed to this, there is the notion of a wound unipotent group introduced by Jacques Tits \cite{tits}: a smooth connected unipotent $k$-group $U$ is said to be {\em $k$-wound} (or just wound when $k$ is clear from the context) if any $k$-morphism $\A^1_k \rightarrow U$ from the affine line to $U$ is the constant morphism to a $k$-point of $U$ \cite[Def.\,B.2.1]{cgp}. This is (non-obviously) equivalent to requiring that there be no non-constant $k$-homomorphism $\Ga \rightarrow U$. Indeed, this property follows from woundness because $\Ga \simeq \A^1$ as $k$-schemes, and conversely, it also implies woundness by \cite[Prop.\,B.3.2]{cgp}. Woundness is also equivalent to requiring that $U$ not contain a $k$-subgroup scheme $k$-isomorphic to $\Ga$, as once again follows from \cite[Prop.\,B.3.2]{cgp}.

Woundness, which may be viewed as the analogue for unipotent groups of anisotropicity for tori (while splitness for unipotent groups is the analogue of splitness for tori), nevertheless behaves in the opposite manner to anisotropicity with regard to field extensions, as we now explain. A $k$-torus becomes split over some finite separable extension (equivalently, every torus over a separably closed field is split), while a smooth connected unipotent group becomes split over a finite {\em purely inseparable} extension (equivalently, every smooth connected unipotent group over a perfect field is split) \cite[Cor.\,B.2.7]{cgp}. Similarly, for a finite purely inseparable extension $k'/k$, a $k$-torus is anisotropic over $k$ if and only if it is so over $k'$, while for a (not necessarily algebraic) separable extension $k''/k$, a smooth connected unipotent $k$-group is wound over $k$ if and only if it is so over $k''$ \cite[Prop.\,B.3.2]{cgp}. Note in particular that the only wound unipotent group over a perfect field is the trivial group, so  woundness is only an interesting notion over imperfect fields (and in particular is uninteresting in characteristic $0$).

Furthermore, woundness enjoys certain permanence properties. It is clearly inherited by smooth connected subgroups, and it is also inherited by extensions: an extension of wound unipotent groups is still wound unipotent, as follows, for example, from the formulation in terms of non-existence of nonzero $k$-homomorphisms from $\Ga$. It is {\em not}, however, inherited by quotients. In fact, every smooth connected unipotent $k$-group $U$ admits an isogeny onto a split group as follows. The group $U_{k^{1/p^n}}$ is split for some $n \geq 0$ (because $U$ splits over a perfect closure of $k$). Let $U^{(p^n)}$ denote the $n$-fold Frobenius twist of $U$ over $k$. We have the following commutative diagram:
\[
\begin{tikzcd}
U^{(p^n)} \arrow{d} \arrow{r}{\sim} \arrow[dr, phantom, "\square"] & U_{k^{1/p^n}} \arrow[dr, phantom, "\square"] \arrow{d} \arrow{r} & U \arrow{d} \\
{\rm{Spec}}(k) \arrow{r}{\sim} & {\rm{Spec}}(k^{1/p^n}) \arrow{r} & {\rm{Spec}}(k) 
\end{tikzcd}
\]
The vertical arrows are the structure maps defining the $k$-scheme structures on $U$ and $U^{(p^n)}$ and the $k^{1/p^n}$-structure on $U_{k^{1/p^n}}$, the first arrow in the bottom row is that induced by the isomorphism $k^{1/p^n} \xrightarrow{\sim} k$ which sends $x$ to $x^{p^n}$, and the second arrow is induced by the inclusion $k \hookrightarrow k^{1/p^n}$. The second square is Cartesian by definition, as is the large outer square. Therefore, so too is the first square. This diagram shows that the $k^{1/p^n}$-splitness of $U_{k^{1/p^n}}$ implies the $k$-splitness of $U^{(p^n)}$. Therefore, the $n$-fold Frobenius map $U \rightarrow U^{(p^n)}$ defines a $k$-isogeny from $U$ onto a split unipotent group.

As we explained previously, woundness is only an interesting notion over imperfect fields. We now give examples which show that, over every imperfect field, the collection of wound unipotent groups is indeed interesting (and in particular, nontrivial).

\begin{example}(\cite[Ex.\,B.1.1]{cgp})
\label{exofwoundgp}
Let $k$ be an imperfect field of characteristic $p$, let $a \in k- k^p$, and consider the subgroup $W_a$ of $\Ga^2 = \Ga \times \Ga$ defined by the following equation:
\[
X + X^p + aY^p = 0.
\]
One may check that the maps $(X, Y) \rightarrow X + a^{1/p}Y$, $T \mapsto (-T^p, a^{-1/p}(T + T^p))$ define mutually inverse isomorphisms between $W_a$ and $\Ga$ over the purely inseparable extension $k(a^{1/p})$. In particular, $W_a$ is smooth connected unipotent. We claim that it is wound. One way to see this is to note that the projective closure $\overline{W}_a$ of $W_a$, defined in $\P^2_k$ by the equation
\[
XZ^{p-1} + X^p + aY^p = 0,
\]
is regular, hence is the regular completion of the affine curve $W_a$. If $W_a$ were not wound, then it would contain, hence be $k$-isomorphic to, $\Ga$. But the regular completion of $\Ga$ is $\P^1_k$, which has rational point at infinity, while the unique point of $\overline{W}_a \backslash W_a$ only becomes rational over $k(a^{1/p})$.
\end{example}

\begin{example}
Here we give an example -- due to Ofer Gabber \cite[Ex.\,2.10]{conradsolvable} -- over any imperfect field of a {\em non-commutative} wound unipotent group. Once again, let $k$ be an imperfect field and let $a \in k - k^p$. Consider the $k$-group $V_a$ defined by the following equation:
\[
X^{p^2} - X + aY^{p^2} = 0.
\]
One may check that the $k$-group $V_a$ becomes isomorphic to $\Ga$ over the purely inseparable extension $k(a^{1/p^2})$. One may also check that the projective closure of $V_a$ is regular with unique point at infinity which becomes rational over $k(a^{1/p^2})$ (exercise!). Therefore, $V_a$ is wound unipotent.

Consider the following nonzero alternating bi-additive map $h: V_a \times V_a \rightarrow W_a$, where $W_a$ is the wound unipotent group from Example \ref{exofwoundgp}:
\[
h\left( (x, y), (x', y')\right) := (xx'^p-x^px', xy'^p-x'y^p).
\]
We let $U_a := W_a \times V_a$ as $k$-schemes, with group law
\[
(w, v)\cdot(w', v') := (w + w' + h(v, v'), v + v').
\]
This defines a group law with identity $(0, 0)$ and inverse $(w, v)^{-1} = (-w, -v)$, and projection onto $V_a$ is a surjective group homomorphism with kernel identified with $W_a$ via the map $w \mapsto (w, 0)$. Further, if $p > 2$ then $U_a$ is non-commutative. It is also wound unipotent because $V_a$ and $W_a$ are. Gabber also gives examples when $p = 2$, but the construction is somewhat more complicated. We refer the reader to \cite[Ex.\,2.10]{conradsolvable}.
\end{example}

This paper is mainly concerned with the existence of moduli spaces representing the functor $\calMor(X, U)$ which sends a $k$-scheme $T$ to the group of $T$-morphisms $X_T \rightarrow U_T$. As is discussed in Appendix \ref{explanationofhypsection}, in order for there to be any hope of representing this functor, we must first of all restrict the underlying category upon which the functor $\calMor(X, U)$ is defined, from the category of all $k$-schemes to that of all geometrically reduced $k$-schemes, and in addition we must impose certain restrictions upon $X$ and $U$: that $X$ be geometrically reduced, and that $U$ not contain a copy of $\Ga$ -- that is, $U$ must be wound. Our main theorem in fact deals with groups somewhat more general than wound unipotent ones, but the wound unipotent case is the main difficulty and will occupy most of our effort in this paper.

\begin{remark}
Before stating the main result, we make a definition in order to eliminate smoothness and connectedness assumptions on the groups we deal with. For a group sheaf $\mathscr{F}$ on the fppf site of some scheme, we define the $n$th derived subsheaf $\mathscr{D}^n\mathscr{F}$ inductively by the formula $\mathscr{D}^0\mathscr{F} := \mathscr{F}$, and for $n \geq 0$, $\mathscr{D}^{n+1}\mathscr{F} := [\mathscr{D}^n\mathscr{F}, \mathscr{D}^n\mathscr{F}]$, where, for subsheaves $\mathscr{G}, \mathscr{G}' \subset \mathscr{F}$, the sheaf $[\mathscr{G}, \mathscr{G}']$ is the one generated by $[\mathscr{G}(R), \mathscr{G}(R')]$ for all $R$ in the site. We say that $\mathscr{F}$ is solvable if $\mathscr{D}^n\mathscr{F} = 1$ for some $n \geq 0$. When $\mathscr{F} = G$ for a smooth connected $k$-group scheme $G$, then the sheaves $\mathscr{D}^nG$ are represented by smooth connected $k$-subgroups schemes of $G$.
\end{remark}

The main result of this paper is the following:

%%% Theorem: Space of morphisms representable
\begin{theorem}
\label{maintheorem}
For a field $k$, a geometrically reduced $k$-scheme $X$ of finite type, and a solvable $k$-group scheme $G$ of finite type not containing a $k$-subgroup scheme $k$-isomorphic to $\Ga$, let $$\calMor(X, G)\colon \{\mbox{$k$-schemes}\} \rightarrow \{\mbox{groups}\}$$ denote the functor defined by the formula $T \mapsto \Mor_T(X_T, G_T)$. Then there is a unique subfunctor $\calMor(X, G)^+ \subset \calMor(X, G)$ with the following two properties:
\begin{itemize}
\item[(i)] The inclusion $\calMor(X, G)^+(T) \subset \calMor(X, G)(T)$ is an equality for all geometrically reduced $k$-schemes $T$.
\item[(ii)] The functor $\calMor(X, G)^+$ is represented by a smooth $k$-group scheme.
\end{itemize}
Furthermore, letting $\calMor(X, G)^+$ also denote the scheme in (ii), then the \'etale component group of $\calMor(X, G)^+$ has finitely-generated group of $k_s$-points. If $G = U$ is wound unipotent, then $\calMor(X, U)^+$ is a smooth, finite type, unipotent $k$-group scheme with wound identity component.
\end{theorem}

The uniqueness of $\calMor(X, G)^+$ follows from (i) and (ii), together with Yoneda's Lemma applied to the category of geometrically reduced $k$-schemes. For an explanation of the hypotheses in Theorem \ref{maintheorem}, see Appendix \ref{explanationofhypsection}. Nontrivial examples of the groups $\calMor(X, U)^+$ described in Theorem \ref{maintheorem} -- or rather of the subgroup of homomorphisms when $X$ is itself an algebraic group -- will be given in Example \ref{Homschemeexample} and Theorem \ref{modulifromcurvesequalsfromjacobians}.

\begin{remark}
\label{mainlyinterestingforperfectfields}
When $k$ is perfect, Chevalley's Theorem says that any smooth connected $k$-group $G$ may be written as an extension
\[
1 \longrightarrow L \longrightarrow G \longrightarrow A \longrightarrow 1
\]
of an abelian variety by a smooth connected affine $k$-group. If $G$ is solvable, then -- because $k$ is perfect -- we have an isomorphism $L \simeq T \ltimes U$ for some $k$-torus $T$ and some split unipotent $k$-group $U$. In particular, if $G$ does not contain a copy of $\Ga$, then $L = T$ is a torus, so $G$ is a semiabelian variety. We therefore see that when $k$ is perfect, Theorem \ref{maintheorem} only applies to semiabelian varieties. As we shall see, however, the most interesting and difficult case of Theorem \ref{maintheorem} is when $G = U$ is wound unipotent, so the main case of interest in Theorem \ref{maintheorem} is when $k$ is imperfect. 
\end{remark}

\begin{remark}
One could consider the restriction of the functor $\calMor(X, G)$ to the category of geometrically reduced $k$-schemes and just assert that this functor is represented by a finite type (necessarily smooth) group scheme. However, the stronger assertion of Theorem \ref{maintheorem}, which describes $\calMor(X, G)^+$ as a subfunctor of $\calMor(X, G)$, is useful because of the deficiencies of the category of geometrically reduced $k$-schemes, as we now explain.

There are various natural assertions which one would wish to be true for the group schemes $\calMor(X, G)^+$. For example, given an inclusion $G \hookrightarrow G'$ of group schemes satisfying the hypotheses of Theorem \ref{maintheorem}, one would hope that the natural map $\phi\colon \calMor(X, G)^+ \rightarrow \calMor(X, G')^+$ induced by the inclusion of the corresponding functors is an inclusion of $k$-group schemes. This does not follow, however, just from the assertion that the map of functors {\em on the category of geometrically reduced $k$-schemes} is an inclusion, because the kernel of a map of smooth group schemes need not be smooth. All we could conclude is that $\ker(\phi)$ is a totally non-smooth group scheme -- that is, that this kernel admits no nonconstant map from a geometrically reduced $k$-scheme (or equivalently, that it has trivial group of $k_s$-points).

But because we know that the functors $\calMor(X, G)^+$ and $\calMor(X, G')^+$ are subfunctors of the functors $\calMor(X, G)$ and $\calMor(X, G')$ on the more robust category of {\em all} $k$-schemes, the commutative diagram
\[
\begin{tikzcd}
\calMor(X, G)^+ \arrow{r}{\phi} \arrow[d, phantom, "\bigcap"] & \calMor(X, G')^+ \arrow[d, phantom, "\bigcap"] \\
\calMor(X, G) \arrow[r, hookrightarrow] & \calMor(X, G')
\end{tikzcd}
\]
shows that $\phi$ is an inclusion of functors on the category of {\em all} $k$-schemes, hence the corresponding map of group schemes is a $k$-group scheme inclusion. There are various assertions of a similar nature that are not at all clear if one merely regards $\calMor(X, G)^+$ as a functor on the category of geometrically reduced $k$-schemes, but which become obvious once one regards $\calMor(X, G)^+$ as a subfunctor of the functor $\calMor(X, G)$ defined on the category of {\em all} $k$-schemes.
\end{remark}

The main work of this paper will lie in proving Theorem \ref{maintheorem} when $G = U$ is wound unipotent. As we explain in Appendix \ref{explanationofhypsection}, the space of morphisms from $X$ into $\Ga$, for example, is usually not representable without imposing some boundedness condition on the morphisms under consideration. The key point in proving Theorem \ref{maintheorem} will therefore be to show that the space of $k_s$-morphisms from $X$ into $U$ is automatically bounded in a suitable sense. (This will turn out to suffice, essentially because a geometrically reduced $k$-scheme of finite type has Zariski dense set of $k_s$-points.) In order to explain the key result that makes everything work, we must first recall the notion of a $p$-polynomial and some attendant facts related to these objects.

A $p$-polynomial over a ring $R$ is a polynomial $F \in R[X_1, \dots, X_n]$ of the form $$F(X_1, \dots, X_n) = \sum_{i \in I \subset \{1, \dots, n\}} \sum_{j = 0}^{N_i} r_{i,\,j}X_i^{p^j}$$ for some subset $I \subset \{1, \dots, n\}$, some $N_i \geq 0$, and some $r_{i,\,N_i} \neq 0$. The $p$-polynomial $\sum_{i\in I} r_{i,\,N_i}X_i^{p^{N_i}}$ is called the {\em principal part} of $F$. The relevance of $p$-polynomials to the study of unipotent groups is that, over a field $k$ of characteristic $p$, {\em every} smooth commutative $p$-torsion $k$-group $U$ is isomorphic to the vanishing locus inside $\Ga^n$ ($n := {\rm{dim}}(U) + 1$) of some smooth $p$-polynomial $F \in k[X_1, \dots, X_n]$ \cite[Prop.\,B.1.13]{cgp}. (Smooth means that the hypersurface defined by $F$ is smooth, or equivalently, that $F$ has nonvanishing linear part.)

Woundness of such $U$ may also be detected at the level of $p$-polynomials. Indeed, a smooth connected commutative $p$-torsion (and therefore unipotent) $k$-group is $k$-wound if and only if it is isomorphic to the vanishing locus of some $p$-polynomial whose principal part has no nontrivial zeroes over $k$ (that is, no zeroes other than $(0, \dots, 0) \in k^n$) \cite[Lem.\,B.1.13, B.1.7]{cgp}. (Not every $p$-polynomial defining a wound group $U$ will have leading part with no nontrivial zeroes, just some defining polynomial will.) Using this fact, the key boundedness result we will show that will allow us to prove Theorem \ref{maintheorem} when $G = U$ is wound unipotent is the following (Proposition \ref{boundedness}): 

%%% Proposition: Boundedness for p-polynomials
\begin{proposition}
\label{boundednessppolynomials}
For any geometrically reduced $k$-scheme $X$ of finite type, and any $p$-polynomial $F \in k[X_1, \dots, X_n]$ whose principal part has no nontrivial zeroes over $k$, consider the map $F_X\colon {\rm{H}}^0(X_{k_s}, \calO_{X_{k_s}})^n \rightarrow {\rm{H}}^0(X_{k_s}, \calO_{X_{k_s}})$ given by $(s_1, \dots, s_n) \mapsto F(s_1, \dots, s_n)$. Then for any finite-dimensional $k$-vector subspace $V \subset {\rm{H}}^0(X_{k_s}, \calO_{X_{k_s}})$, there is a finite-dimensional $k$-vector subspace $W \subset {\rm{H}}^0(X_{k_s}, \calO_{X_{k_s}})$ such that $F_X^{-1}(V) \subset W^n$.
\end{proposition}

To see how Proposition \ref{boundednessppolynomials} may be used to prove the boundedness results we require in order to prove Theorem \ref{maintheorem}, consider the case $V = 0$. Then what we obtain is that the space of all morphisms from $X_{k_s}$ into the unipotent $k$-group $U_{k_s}$ which is the vanishing locus of $F$ has the property that the induced maps $X_{k_s} \rightarrow \mathbf{G}_{a,\,k_s}$ all are defined by global sections of $X$ lying in a finite-dimensional vector space. From this, it is not hard to show that $\calMor(X, U)^+$ may be obtained as the maximal smooth $k$-subgroup scheme of a closed subscheme of some vector group built from this finite-dimensional vector space. Thus the case $V = 0$ of Proposition \ref{boundednessppolynomials} already allows us to prove Theorem \ref{maintheorem} when $G = U$ is wound, commutative, and $p$-torsion. The more general case, dealing with arbitrary finite-dimensional vector spaces $V$ of global sections on $X$, is required in order to deal with arbitrary wound unipotent groups. See the proof of Proposition \ref{mor(x,U)bounded}.

Very roughly, the proof of Proposition \ref{boundednessppolynomials} proceeds as follows. Passing to a dense open subscheme allows us to assume that $X$ is smooth, and that $X$ admits a projective normal compactification $\overline{X}$. We then try to show that a bound on the order of the pole of a section $s \in {\rm{H}}^0(X, \calO_X)$ at some divisor $D \subset \overline{X}$ contained in $\overline{X} \backslash X$ implies a corresponding bound for the orders of the poles at $D$ of sections $s_1, \dots, s_n \in {\rm{H}}^0(X, \calO_X)$ such that $F(s_1, \dots, s_n) = s$.

We first treat the case of curves in \S \ref{curvessection}, where we may take $\overline{X}$ to be the regular compactification of $X$ (which is not generally smooth when $k$ is imperfect!). We then use geometric class field theory in order to prove the statement about orders of poles for the smooth curve $X$, by relating maps from $X$ into $\Ga$ to maps from suitable algebraic groups into $\Ga$; see Theorem \ref{albanese}. To deduce the general case from that of curves, we construct for any divisor $D$ as above a suitable curve fibration $f\colon X' \rightarrow \P^{n-1}$ ($n := {\rm{dim}}(X)$) for a projective scheme $X'$ that admits a birational morphism $\phi\colon X' \rightarrow X$ that is an isomorphism above the complement of a finite subscheme $Z$ of $X$. This fibration is chosen so that $D' := \overline{\phi^{-1}(D - Z)}$, the strict transform of $D$, is a horizontal divisor over $\P^{n-1}$; that is, it intersects the generic fiber of $f$; see Lemma \ref{curvefibrations}. Passing to the generic fiber then allows us to deduce the required boundedness result for orders of poles on $D$ from the case of curves (since the generic fiber is a curve over the function field of $\P^{n-1}$, and one may show that the principal part of $F$ admits no nontrivial zeroes over this function field, because it is separable over $k$).

The structure of this paper is as follows. In \S \ref{curvessection}, we prove the key boundedness result required to prove Theorem \ref{maintheorem} -- namely, Proposition \ref{boundednessppolynomials}. In \S \ref{boundedfamiliessection} we introduce the important notion of a bounded family of morphisms into a smooth connected unipotent group, and show that the space of all $k_s$-morphisms from a finite type geometrically reduced $k$-scheme into a wound unipotent group is bounded. In \S \ref{modulispacesunipsection}, we construct the moduli spaces described in Theorem \ref{maintheorem} when $G$ is wound unipotent. In \S \ref{constructiongmoduligeneralsection}, we prove Theorem \ref{maintheorem}. Finally, in \S \ref{moduliofhomssection}, we use the moduli spaces of morphisms into algebraic groups in order to construct moduli of homomorphisms between algebraic groups, which we then use to show that one of the main results of geometric class field theory remains true after base change by a geometrically reduced scheme; see Corollary \ref{albanesebase}.

%%% Section: Curves
\section{Preimages of spaces of functions under $p$-polynomials}
\label{curvessection}

The purpose of this section is to prove a certain ``boundedness'' result for spaces of functions on geometrically reduced schemes of finite type (Proposition \ref{boundedness}). This result will be the key to proving Theorem \ref{maintheorem} when $G$ is wound unipotent. The proof proceeds by first treating boundedness for curves, and then using a curve fibration lemma (Lemma \ref{curvefibrations}) to treat finite type geometrically reduced $k$-schemes of arbitrary dimension. The argument for curves depends upon geometric class field theory, and in particular the construction of suitable ``Albanese-type'' varieties for maps from smooth curves into commutative groups. Namely, a pointed map from a smooth curve $X$ into a commutative group factors uniquely through a suitable generalized Jacobian of $X$; see Theorem \ref{albanese}. We will use this to understand maps from $X$ into $\Ga$ -- that is, global sections of $X$.

We begin by recalling some facts from geometric class field theory that we shall require. Let $\overline{X}$ be a proper curve over a field $k$ (in our applications, $\overline{X}$ will be regular, though not necessarily smooth), and let $D \subset \overline{X}$ be a finite subscheme. Consider the functor $\{\mbox{$k$-schemes}\} \rightarrow \{\mbox{groups}\}$ which sends a $k$-scheme $T$ to the group of pairs $(\mathscr{L}, \phi)$, where $\mathscr{L} \in \Pic(\overline{X} \times T)$ is a line bundle of relative degree $0$ over $T$, and $\phi\colon \calO_{D_T} \xrightarrow{\sim} \mathscr{L}|_{D_T}$ is a trivialization of $\mathscr{L}$ along $D$. We denote by $\Jac_D(\overline{X})$ the fppf sheafification of this functor. 

\begin{remark}
Although we will never use this, we remark that sheafification is unnecessary if the map $\Gamma(\overline{X}, \calO_{\overline{X}}) \rightarrow \Gamma(D, \calO_D)$ is injective. This injectivity holds, for instance, if $\overline{X}$ is geometrically reduced and geometrically connected (and therefore $\Gamma(\overline{X}, \calO_{\overline{X}}) = k$) and $D \neq \emptyset$. To see why sheafification is unnecessary in this case, we first note that the restriction map on global sections $\Gamma(\overline{X}_T, \calO_{\overline{X}_T}) \rightarrow \Gamma(D_T, \calO_{D_T})$ is injective for any $k$-scheme $T$. Indeed, letting $f\colon \overline{X} \rightarrow {\rm{Spec}}(k)$, $g\colon D \rightarrow {\rm{Spec}}(k)$, and $h\colon T \rightarrow {\rm{Spec}}(k)$ denote the structure maps, we have that $(f_T)_*\calO_{\overline{X}_T} = h^*f_*\calO_{\overline{X}}$ and $(g_T)_*\calO_{D_T} = h^*g_*\calO_D$ because quasi-coherent cohomology commutes with flat base change. Once again using the $k$-flatness of $T$, we therefore see that the injectivity of $(f_T)_*\calO_{\overline{X}_T} \rightarrow (g_T)_*\calO_{D_T}$ (and therefore of the induced map on global sections) follows from that of the map $f_*\calO_{\overline{X}} \rightarrow g_*\calO_D$, and this last injectivity holds by hypothesis.

Now we claim that pairs $(\mathscr{L}, \phi)$ as in the definition of $\Jac_D(\overline{X})$ are rigid -- that is, they admit no nontrivial automorphisms (for the obvious notion of morphism of such a pair, i.e., a morphism of the corresponding line bundles that is compatible with the trivializations). Indeed, an automorphism of $\mathscr{L} \in \Pic(\overline{X}_T)$ is given by a unit $u \in \Gamma(\overline{X}_T, \calO_{\overline{X}_T})^{\times}$. The condition that it be compatible with the trivialization $\phi\colon \calO_{D_T} \xrightarrow{\sim} \mathscr{L}|_{D_T}$ is that $u|_{D_t} = 1$. The injectivity of global sections verified above then implies that $u = 1$ -- that is, the automorphism is the identity.

Finally, to see that sheafification is unnecessary -- i.e., that the presheaf of pairs $(\mathscr{L}, \phi)$ is already an fppf sheaf -- we note that for any fppf covering $\{T_i\}_{i \in I}$ of a $k$-scheme $T$, together with pairs $(\mathscr{L}_i, \phi_i)$ on each $T_i$ and isomorphisms of pairs on double intersections $T_i \times_T T_j$, these isomorphisms automatically satisfy the cocycle condition on triple intersections due to the rigidity proven above. Therefore, they define a descent datum on the $\mathscr{L}_i$, and the $\phi_i$ also descend to give a unique pair $(\mathscr{L}, \phi)$ of the desired sort on $T$.
\end{remark}

The natural map $\Jac_D(\overline{X}) \rightarrow \Jac(\overline{X})$ into the Jacobian of $\overline{X}$ which just forgets the trivialization $\phi$ is surjective as a map of fppf sheaves, since any line bundle -- locally on $T$ -- may be trivialized along the finite scheme $D$. Its kernel consists of pairs $(\calO_{\overline{X}_T}, u)$ with $u \in \Gamma(D_T, \calO_T)^{\times}$, up to isomorphism. Two such pairs $u_1$ and $u_2$ are isomorphic precisely when there is an automorphism of $\calO_{\overline{X}_T}$ mapping $u_1$ to $u_2$, or equivalently, when $u_1/u_2$ extends to an element of $\Gamma(\overline{X}_T, \calO_{\overline{X}_T})^{\times}$. Letting $A := \Gamma(\overline{X}, \calO_{\overline{X}})$, therefore, we obtain a canonical exact sequence
\[
0 \longrightarrow \R_{D/k}(\Gm)/i(\R_{A/k}(\Gm)) \longrightarrow \Jac_D(\overline{X}) \longrightarrow \Jac(\overline{X}) \longrightarrow 0,
\]
where the groups on the left are Weil restrictions of scalars, and $i \colon \R_{A/k}(\Gm) \rightarrow \R_{D/k}(\Gm)$ is the  $k$-homomorphism corresponding to the map $A \rightarrow \Gamma(D, \calO_D)$. In particular, $\Jac_D(\overline{X})$ is smooth. (Note that $i$ is not in general injective.)

If $x \in \overline{X}^{\rm{sm}}(k)$ is a $k$-point lying in the smooth locus of $\overline{X}$, then associated to $x$ one obtains in the usual manner a $k$-morphism $\overline{X}^{\rm{sm}} \rightarrow \Jac(\overline{X})$ by sending a point $y$ to the line bundle associated to the divisor $[y] - [x]$. Similarly, if $x \notin D$, then one obtains a map $i_x\colon \overline{X}^{\rm{sm}}\backslash D \rightarrow \Jac_D(\overline{X})$ via the map sending $y$ to the pair $(\calO([y]-[x]), \phi)$, where $\phi$ is the canonical trivialization along $D$ of the line bundle associated to the divisor $[y] - [x]$ which is disjoint from $D$. The main result from geometric class field theory that we require says that these maps have an Albanese type property with respect to maps from smooth curves into commutative algebraic groups.

%%% Theorem: Maps from curves to commutative groups factor through generalized Jacobians
\begin{theorem}[{\rm{\cite[\S 10.3, Thm.\,2]{neronmodels}}}]
\label{albanese}
Let $X$ be a smooth curve over a field $k$, with regular compactification $\overline{X}$, and boundary $D := \overline{X}\backslash X$, and let $x \in X(k)$. Then for any finite type commutative $k$-group scheme $G$, and any morphism of pointed $k$-schemes $f\colon (X, x) \rightarrow (G, 0)$, there exist a divisor $D' \subset \overline{X}$ with support $D$ and a unique $($for a given $D'$$)$ $k$-group scheme homomorphism $\psi\colon \Jac_{D'}(\overline{X}) \rightarrow G$ such that the following diagram commutes:
\[
\begin{tikzcd}
X \arrow{r}{i_x} \arrow{dr}[swap]{f} & \Jac_{D'}(\overline{X}) \arrow[d, dashrightarrow, "\exists! \psi"] \\
& G
\end{tikzcd}
\]
\end{theorem}

The statement in \cite[\S 10.3, Thm.\,2]{neronmodels} includes the assumption that $G$ is smooth, but this is unnecessary because the smoothness of $X$ ensures that any $k$-morphism $X \rightarrow G$ lands inside the maximal smooth $k$-subgroup scheme of $G$ \cite[Lem.\,C.4.1, Rem.\,C.4.2]{cgp}, and similarly for maps $\Jac_{D'}(\overline{X}) \rightarrow G$, hence the theorem immediately reduces to the smooth case. 

We will use Theorem \ref{albanese} when $G = \Ga$ in order to prove the key boundedness result for curves (Proposition \ref{boundednessforcurves}). We first require a result explaining how the groups $\Jac_D(\overline{X})$ vary as we thicken the divisor $D$. Note that, by restricting the trivialization, we obtain for any pair of divisors $D \subset D'$ on $\overline{X}$ a natural surjective $k$-homomorphism $\Jac_{D'}(\overline{X}) \twoheadrightarrow \Jac_D(\overline{X})$.

%%% Lemma: Jac_D(X) expands by a split unipotent group as one expands the divisor D
\begin{lemma}
\label{thickendivisor}
Let $\overline{X}$ be a proper curve over a field $k$, and let $D \subset D'$ be Weil divisors on $\overline{X}$ with the same support. Then the kernel of the natural surjective map $\Jac_{D'}(\overline{X}) \twoheadrightarrow \Jac_D(\overline{X})$ is split unipotent.
\end{lemma}

\begin{proof}
It suffices to treat the case when $D' = D + x$ for some closed point $x \in {\rm{Supp}}(D)$. Let $\mathfrak{m}_x \subset \calO_{\overline{X}, x}$ denote the maximal ideal of the local ring at $x$, and suppose that $x$ appears in $D$ with multiplicity $n > 0$. In this case, the kernel in question admits a natural surjective map from the functor sending a $k$-scheme $T$ to $1 + \mathfrak{m}_x^n\Gamma(T, \calO_T) \pmod{\mathfrak{m}_x^{n+1}}$. Indeed, the map is the one sending $u$ to the pair $(\calO_{\overline{X}}, u)$. (The reason that the kernel is only a quotient of this group in general is that we have to quotient by the global units on $\overline{X}$ that are $1 \pmod{\mathfrak{m}_x^n}$, since two pairs which differ by such a global unit will be isomorphic.) It therefore suffices to check that this functor is represented by a split unipotent $k$-group. The map $f \mapsto 1 + f$ is an isomorphism from $(\mathfrak{m}_x^n/\mathfrak{m}_x^{n+1}, +)$ to $(1 + \mathfrak{m}_x^n \pmod{\mathfrak{m}_x^{n+1}}, \cdot)$, and similarly for $\mathfrak{m}_x^n\Gamma(T, \calO_T)$, etcetera. Hence we see that the functor in question is represented by the vector group associated to the finite-dimensional $k$-vector space $\mathfrak{m}_x^n/\mathfrak{m}_x^{n+1}$.
\end{proof}

We will also require the following result.

%%% Proposition: Finite Generation of Hom(G, G_a)
\begin{proposition}
\label{homfg}
For a smooth connected group scheme $G$ over a field $k$, $\Hom(G, \Ga)$ is a finitely-generated $\End(\Ga)$-module.
\end{proposition}

One can relax the conditions on $G$ in the above result, but it will suffice for our purposes.

\begin{proof}
We have $\Hom(G, \Ga) = \Hom(G/\mathscr{D}G, \Ga)$, where $\mathscr{D}G$ is the derived group of $G$. We may therefore assume that $G$ is commutative. Then $G$ admits a filtration by semiabelian varieties and smooth connected unipotent groups: over perfect fields this follows from Chevalley's Theorem, while over imperfect fields (or any field of characteristic $p$), it follows from \cite[Th.\,A.3.9]{cgp}. We therefore reduce to these cases. In the semiabelian variety case, $\Hom(G, \Ga) = 0$, so assume that $G = U$ is smooth connected commutative unipotent. If ${\rm{char}}(k) = 0$, then it follows that $U \simeq \Ga^n$ \cite[Prop.\,14.32]{milne}, and we are reduced to the trivial case $U = \Ga$. If ${\rm{char}}(k) = p > 0$, then $U$ is killed by some $p^n$. Then applying the exact sequence
\[
0 \longrightarrow [p]U \longrightarrow U \longrightarrow U/[p]U \longrightarrow 0
\]
and induction, we may assume that $U$ is $p$-torsion. In this case, $U$ is a $k$-subgroup scheme of some $\Ga^n$ (\cite[Prop.\,B.1.13]{cgp} as well as the remark immediately following to handle finite $k$). The result in this case is therefore a special case of \cite[Ch.\,IV, \S3, Cor.\,6.7]{demazuregabriel}.
\end{proof}

Let us recall a definition.

\begin{definition}
\label{principalpart}
For a $p$-polynomial $F(X_1, \dots, X_n) = \sum_{i \in I \subset \{1, \dots, n\}} \sum_{j = 0}^{N_i} c_{i,\,j}X_i^{p^j} \in k[X_1, \dots, X_n]$ with $N_i \geq 0$ and $c_{i, N_i} \neq 0$ for all $i$, we call $$\sum_{i \in I} c_{i,\,N_i}X_i^{p^{N_i}} \in k[X_1, \dots, X_n]$$ the {\em principal part} of $F$. We say that the principal part of $F$ has no nontrivial zeroes over $k$ if its only zero in $k^n$ is $\mathbf{0} \in k^n$.
\end{definition}

We are now prepared to prove the crucial boundedness result required to prove Theorem \ref{maintheorem} when $X$ is a curve.

%%% Proposition: Boundedness for Curves
\begin{proposition}
\label{boundednessforcurves}$($Boundedness for curves$)$
Let $X$ be a geometrically reduced everywhere $1$-dimensional finite type scheme over a field $k$ of characteristic $p > 0$, $F \in k[X_1, \dots, X_n]$ a $p$-polynomial whose principal part has no nontrivial zeroes over $k$, and let $F_X\colon {\rm{H}}^0(X, \calO_X)^n \rightarrow {\rm{H}}^0(X, \calO_X)$ be the map $(s_1, \dots, s_n) \mapsto F(s_1, \dots, s_n)$. Then for any finite-dimensional $k$-vector space $V \subset {\rm{H}}^0(X, \calO_X)$, there is a finite-dimensional $k$-vector space $W \subset {\rm{H}}^0(X, \calO_X)$ such that $F_X^{-1}(V) \subset W^n$.
\end{proposition}

\begin{proof}
We first note that we are free to replace $X$ with a dense open subscheme $Y$. Indeed, if we let $V|_Y$ denote $\{v|_Y \mid v \in V\}$, then assuming that the proposition holds for $Y$, we find that $\{(f_1, \dots, f_n) \in {\rm{H}}^0(Y, \calO_Y)^n \mid F(f_1, \dots, f_n) \in V|_Y\}$ spans a finite-dimensional $k$-vector space. Since the map ${\rm{H}}^0(X, \calO_X) \rightarrow {\rm{H}}^0(Y, \calO_Y)$ is injective, it follows that $\{(f_1, \dots, f_n) \in {\rm{H}}^0(Y, \calO_Y)^n \mid F(f_1, \dots, f_n) \in V\}$ is also finite-dimensional, as required. We may therefore assume that $X$ is smooth and quasi-projective (affine even). We are also free to pass from $k$ to $k_s$, since this preserves the hypothesis that the principal part of $F$ has no nontrivial zeroes \cite[Lem.\,B.1.8]{cgp}. In particular, we may assume that $X(k) \neq \emptyset$. Choose a point $x \in X(k)$.

We claim that it also suffices to prove the proposition with ${\rm{H}}^0(X, \calO_X)$ replaced by ${\rm{H}}^0(X, \calO(-x))$, the space of functions vanishing at $x$. Indeed, we are free to enlarge $V$, so we may assume that $V$ contains the constant functions. If the proposition holds for ${\rm{H}}^0(X, \calO(-x))$, then there is a finite-dimensional $k$-vector space $W \subset {\rm{H}}^0(X, \calO(-x))$ such that
\[
\{(f_1, \dots, f_n) \in {\rm{H}}^0(X, \calO(-x))^n \mid F(f_1, \dots, f_n) \in V \cap {\rm{H}}^0(X, \calO(-x))\} \subset W^n.
\]
We then claim that
\[
\{(f_1, \dots, f_n) \in {\rm{H}}^0(X, \calO_X)^n \mid F(f_1, \dots, f_n) \in V\} \subset (W + k)^n.
\]
Indeed, suppose that $F(f_1, \dots, f_n) \in V$. Let $c := F(f_1(x), \dots, f_n(x)) \in k \subset V$. Then $f_i-f_i(x) \in {\rm{H}}^0(X, \calO(-x))$, and $$F(f_1-f_1(x), \dots, f_n - f_n(x)) = F(f_1, \dots, f_n) - F(f_1(x), \dots, f_n(x)) \in V \cap {\rm{H}}^0(X, \calO(-x)),$$ so each $f_i-f_i(x) \in W$, hence $f_i \in W + k$. Summarizing, $X$ is a smooth quasi-projective curve, $x \in X(k)$, and we are free to prove the proposition with ${\rm{H}}^0(X, \calO_X)$ replaced by ${\rm{H}}^0(X, \calO(-x))$.

Let $\overline{X}$ be the regular compactification of $X$, and let $D := \overline{X}\backslash X$ with its reduced structure. An element of ${\rm{H}}^0(X, \calO(-x))$ is the same as a morphism of pointed $k$-schemes $f\colon (X, x) \rightarrow (\Ga, 0)$. By Theorem \ref{albanese}, for every such morphism there exist a positive integer $N$ and a unique (once one has chosen $N$) $k$-group homomorphism $\psi\colon \Jac_{ND}(\overline{X}) \rightarrow \Ga$ such that the following diagram commutes:
\begin{equation}
\label{boundednessforcurvespfeqn1}
\begin{tikzcd}
X \arrow{r}{i_{x,\,N}} \arrow{dr}[swap]{f} & \Jac_{ND}(\overline{X}) \arrow[d, dashrightarrow, "\exists! \psi"] \\
& \Ga
\end{tikzcd}
\end{equation}
We claim that there is a positive integer $M > 0$ with the following property: for every integer $M' \geq M$, if $\psi\colon \Jac_{M'D}(\overline{X}) \rightarrow \Ga$ is a homomorphism that does not factor through the canonical surjection $\Jac_{M'D}(\overline{X}) \twoheadrightarrow \Jac_{MD}(\overline{X})$, then $\psi \circ i_{x,\,M'} \notin V$. Indeed, thanks to the uniqueness of the homomorphism furnished by Theorem \ref{albanese}, it suffices to show that there is an $M$ such that, in the diagram (\ref{boundednessforcurvespfeqn1}), we may take $N = M$ for every $f \in V$. In fact, since $V$ is finite-dimensional, we may choose $M$ that works for every element of some chosen basis for $V$. Since a $k$-linear combination of homomorphisms $\Jac_{MD}(\overline{X}) \rightarrow \Ga$ is still such a homomorphism, it follows that this same $M$ works for all $f \in V$, as desired.

Now suppose that $f_1, \dots, f_n\colon (X, x) \rightarrow (\Ga, 0)$ are such that $F(f_1, \dots, f_n) \in V$. Then we claim that, for each $f_i$, there is a commutative diagram
\begin{equation}
\label{boundednessforcurvespfeqn2}
\begin{tikzcd}
X \arrow{r}{i_{x,\,M}} \arrow{dr}[swap]{f_i} & \Jac_{MD}(\overline{X}) \arrow[d, dashrightarrow, "\exists! \psi_i"] \\
& \Ga
\end{tikzcd}
\end{equation}
with $M$ the positive integer of the preceding paragraph. Indeed, Theorem \ref{albanese} implies that there is such a diagram, but with $M$ perhaps replaced by a larger integer $N > M$. Let $U := \ker(\Jac_{ND}(\overline{X}) \rightarrow \Jac_{MD}(\overline{X}))$. By Lemma \ref{thickendivisor}, $U$ is split unipotent. If, for at least one $f_i$, the map $\psi_i$ does not factor through $\Jac_{MD}(\overline{X})$, then the restricted homomorphism $\psi_i|_U \colon U \rightarrow \Ga$ is nonzero. By \cite[Lem.\,B.1.7]{cgp}, the fact that the principal part of $F$ has no nontrivial zeroes over $k$ implies that $\ker(F\colon \Ga^n \rightarrow \Ga)$ admits no nonzero $k$-homomorphism from a split unipotent group. Therefore, the non-vanishing of some some $\psi_i|U$ would imply that $F(\psi_1|_U, \dots, \psi_n|_U) \neq 0$, hence that $F(\psi_1, \dots, \psi_n) \colon \Jac_{ND}(\overline{X}) \rightarrow \Ga$ does not factor through $\Jac_{MD}(\overline{X})$. But $F(f_1, \dots, f_n) \in V$, so this contradicts our choice of $M$. This proves the claim that whenever $F(f_1, \dots, f_n) \in V$, necessarily there is a diagram (\ref{boundednessforcurvespfeqn2}) for a fixed (i.e., depending only on $V$) positive integer $M$. 

By Proposition \ref{homfg} applied to $G := \Jac_{MD}(\overline{X})$, we see that there are finitely many sections $s_1, \dots, s_r \in {\rm{H}}^0(X, \calO(-x))$ such that, for every $n$-tuple $f_1, \dots, f_n \in {\rm{H}}^0(X, \calO(-x))$ satisfying $F(f_1, \dots, f_n) \in V$, one has that each $f_i$ is a $p$-polynomial in $s_1, \dots, s_r$. Let $y_1, \dots, y_m$ be the closed points of $D$. We will show that the poles of the $f_i$ at each $y_j$ are of uniformly bounded order (i.e., depending only on $V$ and not on the $f_i$). First, this holds for elements of $V$ because $V$ is finite-dimensional. Now let $y$ be one of the $y_j$, and let $k' := k(y)$ be the residue field of $\overline{X}$ at $y$, which is a finite purely inseparable extension of the separably closed field $k$. Choose a uniformizer $T$ of the local ring $\calO_{\overline{X}, y}$ of $\overline{X}$ at $y$. For an element $g \in {\rm{Frac}}(\calO_{\overline{X}, y})$ lying in the fraction field of the local ring at $y$, we say that $g$ has {\em leading term} $\alpha T^r$ if $T^{-r}g \in \calO_{\overline{X}, y}$ and if $\alpha \in k'$ is its image under the reduction map (which is a $k$-algebra map) $\calO_{\overline{X}, y} \rightarrow k'$. (Note that a leading term can be $0$, so in particular any element has many different leading terms corresponding to those $r$ which are $\leq {\rm{ord}}_y(g)$. A nonzero element has a unique leading term with $r = {\rm{ord}}_y(g)$.) The leading terms admit the expected formal properties under addition and multiplication of elements.

Without loss of generality, suppose that $s_i$ is not regular at $y$ for $1 \leq i \leq e$, and that it is regular at $y$ for $e < i \leq r$. Let $n_i$ be the order of the pole of $s_i$ at $y$ ($1 \leq i \leq e$), so that $s_i$ has leading term $\alpha_iT^{-n_i}$ for some $\alpha_i \in (k')^{\times}$. Then there is some integer $b' \geq 0$ such that, for every $b \geq b'$ and every $1 \leq i \leq e$, $\alpha_i^{p^b} \in k^{\times}$, and therefore $s_i^{p^b}$ has leading term lying in $(k^{\times})T^{-n_ip^b}$. We see, therefore, that there is a positive integer $C$ such that, if a $p$-polynomial $f$ in the $s_i$ has pole of order $C' \geq C$ at $y$, then the leading term of $f$ lies in $(k^{\times})T^{-C'}$. It then follows that, if $f_1, \dots, f_n$ are $p$-polynomials in the $s_i$ (as they must be if they vanish at $x$ and satisfy $F(f_1, \dots, f_n) \in V$) at least one of which has a pole of order $\geq C'$ at $y$, then $F(f_1, \dots, f_n)$ has leading term $P(\beta_1, \dots, \beta_n)T^{-C''}$ for $P$ the principal part of $F$, $C'' := \max_i\{-{\rm{deg}}_{X_i}(F){\rm{ord}}_y(f_i)\} \geq C'$, and $\beta \in k$ not all $0$. Because the principal part of $F$ has no nontrivial zeroes over $k$, it follows that ${\rm{ord}}_y(F(f_1, \dots, f_n)) = -C'' \leq - C'$. That is, $F(f_1, \dots, f_n)$ also has a pole of order $\geq C'$. In particular, if $C'$ is chosen large enough, then $F(f_1, \dots, f_n) \notin V$. We therefore see that there is a uniform bound on the order of the poles at $y$ of functions $f_1, \dots, f_n \in {\rm{H}}^0(X, \calO(-x))$ such that $F(f_1, \dots, f_n) \in V$. Since this holds for all $y \in \overline{X}\backslash X$, we see that the $f_i$ are restricted to lie in a finite-dimensional $k$-vector space, as desired.
\end{proof}

Now we proceed to the proof of Proposition \ref{boundednessforcurves} (``boundedness for curves'') for geometrically reduced $k$-schemes of arbitrary dimension. The idea is to use suitable curve fibrations to reduce to the curve case. The key lemma is the following.

%%% Lemma: Curve Fibrations
\begin{lemma}
\label{curvefibrations}
Let $k$ be an infinite field, $\overline{X}$ a projective, geometrically irreducible, generically smooth $k$-scheme of dimension $d > 0$, and $D \subset \overline{X}$ an irreducible divisor. Then there is a projective $k$-scheme $\overline{X'}$ equipped with $k$-morphisms $\phi\colon \overline{X'} \rightarrow \overline{X}$ and $f\colon \overline{X'} \rightarrow \P^{d-1}$ such that
\begin{itemize}
\item[(i)] there is a finite $k$-subscheme $Z \subset \overline{X}$ such that $\phi$ restricts to an isomorphism \newline $\phi^{-1}(\overline{X} \backslash Z) \xrightarrow{\sim} \overline{X}\backslash Z$;
\item[(ii)] the map $f$ is smooth on a dense open subscheme of $\overline{X'}$, and its generic fiber is a curve;
\item[(iii)] if $D \not \subset Z$, then letting $D' := \overline{\phi^{-1}(D\backslash (D \cap Z))}$ denote the strict transform of $D$, $D'$ is a ``horizontal'' divisor with respect to $f$; that is, $D'$ intersects the generic fiber of $f$.
\end{itemize}
\end{lemma}

\begin{proof}
Choose an embedding $\overline{X} \subset \P^N$. For hyperplanes $H_0, \dots, H_{d-1}$, with $H_i$ defined by the equation $F_i := \sum_{j=0}^N c_{i,\,j}Z_j = 0$, consider the rational map $\pi \colon \P^N \dashrightarrow \P^{d-1}$ whose $i$th coordinate is $F_i$. Bertini's Theorem says that, in the space of all $H_0, \dots, H_{d-1}$, there is a nonempty open subscheme satisfying the following three conditions:
\begin{itemize}
\item[(i)] $\overline{X} \cap (H_0 \cap \dots \cap H_{d-1})$ is finite;
\item[(ii)] $\overline{X} \cap (H_1 \cap \dots \cap H_{d-1})$ is a generically smooth curve, and $H_1 \cap \dots \cap H_{d-1}$ intersects $\overline{X}$ transversely on the smooth locus of $\overline{X}$; and
\item[(iii)] $D \cap (H_1 \cap \dots \cap H_{d-1})$ is finite and nonempty.
\end{itemize}
Since the space of all $d$-tuples of hyperplanes is rational and $k$ is infinite, it follows that one may choose such hyperplanes $H_0, \dots, H_{d-1}$ defined over $k$.

Condition (i) says that $\pi|_{\overline{X}} \colon \overline{X} \dashrightarrow \P^{d-1}$ defines a morphism on the complement of a finite subscheme of $\overline{X}$. Condition (ii) implies that the fiber above $[1: 0: \dots : 0]$ is a generically smooth curve, and the fact that the intersection is generically transverse says that the map $\pi|_{\overline{X}}$ is actually smooth at the generic point of this fiber, hence $\pi|_{\overline{X}}$ is generically smooth. Furthermore, the fact that the fiber above $[1: 0: \dots : 0]$ is a curve, together with the upper semicontinuity of fiber dimension and the fact that $\overline{X}$ is of dimension $1$ greater than $\P^{d-1}$, implies that the generic fiber of $\pi|_{\overline{X}}$ is a curve. Finally, similar considerations with dimension, together with the fact that the fiber of $\pi|_D$ above $[1: 0: \dots : 0]$ is finite and nonempty, tell us that $D$ is a horizontal divisor with respect to the rational map $\pi|_{\overline{X}}$. To obtain the scheme $\overline{X'}$ with the required properties, simply blow up $\overline{X}$ along the finite subscheme on which $\pi|_{\overline{X}}$ fails to be defined in order to resolve the locus of indeterminacy of the rational map $\pi|_{\overline{X}}$ (see \cite[Ch.\,II, Example 7.17.3]{hartshorne}).
\end{proof}

We also require the following lemma.

%%% Lemma: Principal part has no nontrivial zeroes if and only if this holds over a separable extension
\begin{lemma}
\label{nontrivialzeroessepble}
Let $k$ be a field of characteristic $p > 0$, and let $K/k$ be a separable extension field. Then for a $p$-polynomial $H \in k[X_1, \dots, X_n]$ of the form $H = \sum_{i=1}^n c_iX_i^{p^{n_i}}$, $H$ has no nontrivial zeroes over $k$ if and only if it has no nontrivial zeroes over $K$.
\end{lemma}

\begin{proof}
The if direction is trivial. For the only if direction, suppose that $F$ admits a nontrivial zero over $K$. Writing $K$ as a filtered direct limit of smooth $k$-algebras (possible because $K/k$ is separable), this zero spreads out to a nontrivial zero $\mathbf{z} \in \Gamma(X, \calO_X)^n$ over some nonempty smooth $k$-scheme $X$. Then for some $x \in X(k_s)$, $\mathbf{z}_x \neq \mathbf{0}$ (since $X(k_s)$ is Zariski dense in $X$). We thereby obtain a nontrivial zero of $H$ over $k_s$. That is, we have reduced the lemma to the case $K = k_s$, and this case follows from \cite[Lem.\,B.1.8]{cgp}.
\end{proof}

We are now prepared to prove the key boundedness result for higher-dimensional schemes.

%%% Proposition: Boundedness in general
\begin{proposition}$($``Boundedness''$)$
\label{boundedness}
Let $k$ be a field of characteristic $p > 0$, $F \in k[X_1, \dots, X_n]$ a $p$-polynomial whose principal part has no nontrivial zeroes over $k$, $X$ a geometrically reduced $k$-scheme of finite type, and $V \subset {\rm{H}}^0(X, \calO_X)$ a finite-dimensional $k$-vector space. Let $F_X\colon {\rm{H}}^0(X, \calO_X)^n \rightarrow {\rm{H}}^0(X, \calO_X)$ denote the map $(s_1, \dots, s_n) \mapsto F(s_1, \dots, s_n)$. Then there is a finite-dimensional $k$-vector space $W \subset {\rm{H}}^0(X, \calO_X)$ such that $F_X^{-1}(V) \subset W^n$.
\end{proposition}

\begin{proof}
We are free to extend scalars from $k$ to $k_s$, as this preserves the hypothesis that the principal part of $F$ has no nontrivial zeroes \cite[Lem.\,B.1.8]{cgp}. Furthermore, by the same argument as in the beginning of the proof of Proposition \ref{boundednessforcurves}, we may replace $X$ by a dense open subscheme. We are therefore free to assume that the irreducible and connected components of $X$ agree. We may replace $X$ by its connected components and thereby assume that $X$ is geometrically integral. When $X$ is $0$-dimensional, it is just ${\rm{Spec}}(k)$ and the assertion is trivial. When $X$ is $1$-dimensional, the proposition is Proposition \ref{boundednessforcurves}. We may therefore assume that $d := {\rm{dim}}(X) > 1$.

We may also (by shrinking $X$) assume that $X$ is smooth and quasi-projective, hence embeds as a dense open subscheme of a normal, generically smooth, geometrically irreducible, projective $k$-scheme $\overline{X}$. In order to show that $F_X^{-1}(V)$ lies in a finite-dimensional space, it suffices to show that, for each of the finitely many divisors $D$ of $\overline{X}$ lying in $\overline{X}\backslash X$, there is a uniform (i.e., depending only on $V$) bound on the orders of the poles along $D$ of functions $f_1, \dots, f_n \in {\rm{H}}^0(X, \calO_X)$ such that $F(f_1, \dots, f_n) \in V$ (because, if $\partial X$ is the union of the divisors lying in $\overline{X}\backslash X$, then ${\rm{H}}^0(\overline{X}, \calO_{\overline{X}}(N\cdot \partial X))$ is finite-dimensional for any integer $N$).

So let $D \subset \overline{X}$ be a divisor lying in the complement of $X$. Applying Lemma \ref{curvefibrations}, we obtain a projective $k$-scheme $\overline{X'}$ and $k$-morphisms $\phi\colon \overline{X'} \rightarrow \overline{X}$ and $f\colon \overline{X'} \rightarrow \P^{d-1}$ such that the following three conditions hold:
\begin{itemize}
\item[(i)] there is a finite $k$-subscheme $Z \subset \overline{X}$ such that $\phi$ restricts to an isomorphism \newline $\phi^{-1}(\overline{X} \backslash Z) \xrightarrow{\sim} \overline{X}\backslash Z$;
\item[(ii)] the map $f$ is smooth on a dense open subscheme of $\overline{X'}$, and its generic fiber is a curve;
\item[(iii)] letting $D' := \overline{\phi^{-1}(D\backslash (D \cap Z))}$ denote the strict transform of $D$, then $D'$ is a ``horizontal'' divisor with respect to $f$; that is, it intersects the generic fiber of $f$.
\end{itemize}
(Note that the condition $D \not \subset Z$ given in Lemma \ref{curvefibrations}(iii) is automatically satisfied here because ${\rm{dim}}(X) > 1 \Longrightarrow {\rm{dim}}(D) > 0$.) Let $\eta$ be the generic point of $\P^{d-1}$. Then $\overline{X'}_\eta$ is a projective, generically smooth curve over $k(\eta)$ by (ii). Furthermore, because $X$ is regular in codimension $1$, the same holds for $\overline{X'}\backslash \phi^{-1}(Z)$ by (i). Because $D' \not \subset \phi^{-1}(Z)$, therefore, it follows that $\overline{X'}$ is regular at the generic point of $D$, which lies in $\overline{X'}_\eta$ by (iii). Therefore, it makes sense to speak of the order of a function on $\overline{X'}_\eta$ along $D_\eta$, and this order agrees with the order of the corresponding rational function on $\overline{X}$ along $D$. Furthermore, by Lemma \ref{nontrivialzeroessepble}, the principal part of $F$ still has no nontrivial zeroes over $k(\eta)$. Applying Proposition \ref{boundednessforcurves} to the $k(\eta)$-span of the finite-dimensional $k$-vector space $V|_{\overline{X'}_\eta}$, therefore, we find that there is a uniform bound $N > 0$ such that, for any functions $f_1, \dots, f_n \in {\rm{H}}^0(X, \calO_X)$ satisfying $F(f_1, \dots, f_n) \in V$, the $f_i$ all have poles of order $\leq N$ along $D$, as desired. This completes the proof of the proposition.
\end{proof}

%%% Section: Bounded Families of Morphisms
\section{Bounded families of morphisms}
\label{boundedfamiliessection}

In this section, we will introduce the notion of a bounded family of morphisms from a $k$-scheme $X$ into a smooth connected unipotent $k$-group $U$, and show that, when $X$ is geometrically reduced and $U$ is wound, the family of $k_s$-morphisms from $X$ into $U$ is bounded. This is the key result that will allow us to construct the moduli spaces $\calMor(X, U)^+$ in \S \ref{modulispacesunipsection}. The key to showing this boundedness result is Proposition \ref{boundedness}.

Given a smooth connected unipotent group $U$ over a field $k$, one has a $\overline{k}$-scheme isomorphism $\overline{\phi}\colon U_{\overline{k}} \xrightarrow{\sim} \A^d_{\overline{k}}$, where $d := {\rm{dim}}(U)$. Indeed, this holds for split unipotent groups over $k$ itself, and every smooth connected unipotent group becomes split over $\overline{k}$ -- in fact, over $k_{\perf}$ \cite[Thm.\,15.4(iii)]{borelalggroups}. Given a $k$-scheme $X$, we say that a family $\mathscr{F}$ of $T$-morphisms $X_T \rightarrow U_T$ ($T$ a $k$-scheme) is {\em $\overline{\phi}$-bounded} if there is a finite-dimensional $k$-vector space $V \subset {\rm{H}}^0(X, \calO_X)$ such that $\{ \overline{\phi} \circ f_{\overline{k}} \mid f \in \mathscr{F}\} \subset V^d \otimes_k {\rm{H}}^0(T_{\overline{k}}, \calO_{T_{\overline{k}}}) \subset {\rm{H}}^0(X, \calO_X) \otimes_k {\rm{H}}^0(T_{\overline{k}}, \calO_{T_{\overline{k}}}) \subset {\rm{H}}^0(X_{T_{\overline{k}}}, \calO_{X_{T_{\overline{k}}}})$ (the last containment is an equality if $T$ is quasi-compact).

%%% Lemma: Boundedness is independent of isomorphism with \A^d
\begin{lemma}
\label{boundednessindofisom}
Given $k$-schemes $X$ and $T$, a family $\mathscr{F}$ of $T$-morphisms $X_T \rightarrow U_T$, and two $\overline{k}$-scheme isomorphisms $\overline{\phi}, \overline{\psi}\colon U_{\overline{k}} \xrightarrow{\sim} \A^d_{\overline{k}}$, one has that $\mathscr{F}$ is $\overline{\phi}$-bounded if and only if it is $\overline{\psi}$-bounded.
\end{lemma}

\begin{proof}
Suppose that $\mathscr{F}$ is $\overline{\phi}$-bounded, and we will show that it is $\overline{\psi}$-bounded. Let $\overline{\alpha} := \overline{\psi}\circ \overline{\phi}^{-1}$, a $\overline{k}$-automorphism of $\A^d_{\overline{k}}$. Then we have a commutative diagram
\begin{equation}
\label{boundednessindofisompfeqn1}
\begin{tikzcd}
U_{\overline{k}} \arrow{r}{\overline{\phi}} \arrow{dr}[swap]{\overline{\psi}} & \A^d_{\overline{k}} \arrow{d}{\overline{\alpha}} \\
& \A^d_{\overline{k}}
\end{tikzcd}
\end{equation}
Choose a finite-dimensional $k$-vector space $V \subset {\rm{H}}^0(X, \calO_X)$ such that $\overline{\phi} \circ \mathscr{F}_{\overline{k}} := \{ \overline{\phi} \circ f_{\overline{k}} \mid f \in \mathscr{F}\} \subset V^d \otimes_k {\rm{H}}^0(T_{\overline{k}}, \calO_{X_{T_{\overline{k}}}})$. Then, via $\overline{\alpha}$, $V^d$ is mapped into the $\Gamma(T_{\overline{k}}, \calO_{_{\overline{k}}})$-span of a finite-dimensional $k$-subspace of ${\rm{H}}^0(X, \calO_X)^d$. The commutative diagram (\ref{boundednessindofisompfeqn1}) then implies that $\mathscr{F}$ is $\overline{\psi}$-bounded.
\end{proof}

In light of Lemma \ref{boundednessindofisom}, we may make the following definition.

%%% Definition: Boundedness of a family of morphisms
\begin{definition}
\label{boundedfamiliesofmorphismsdef}
Given schemes $X$ and $T$ over a field $k$, a smooth connected unipotent $k$-group $U$, and a family $\mathscr{F}$ of $T$-morphisms $X_T \rightarrow U_T$, we say that $\mathscr{F}$ is {\em bounded} if it is $\overline{\phi}$-bounded for some (equivalently, every) $\overline{k}$-scheme isomorphism $\overline{\phi}\colon U_{\overline{k}} \xrightarrow{\sim} \A^d_{\overline{k}}$.
\end{definition}

\begin{remark}
\label{boundedfork_scriterion}
The case that we shall be interested in is when $T = {\rm{Spec}}(k_s)$. Note that, in this case, an equivalent definition of boundedness is that $\{f_{\overline{k}} \mid f \in \mathscr{F}\} \subset \Mor_{\overline{k}}(X_{\overline{k}}, U_{\overline{k}})$ is bounded in the sense of being contained (via an isomorphism $\overline{\phi}$) in $V^d \otimes_{k} \overline{k}$ for some finite-dimensional $k$-vector space $V \subset {\rm{H}}^0(X, \calO_X)$. A similar remark applies whenever $T = {\rm{Spec}}(k')$ for some algebraic extension field $k'/k$.
\end{remark}

We now prove various permanence properties for boundedness of families of morphisms under certain maps between unipotent groups.

%%% Lemma: Boundedness may be checked after passing to a larger group, and is preserved by surjections
\begin{lemma}
\label{boundednesslargergroup}
Suppose given a $k$-scheme morphism $h\colon U_1 \rightarrow U_2$ between smooth connected unipotent $k$-group schemes, and $k$-schemes $X$ and $T$.
\begin{itemize}
\item[(i)] If $\mathscr{G}$ is a bounded family of $T$-morphisms $X_T \rightarrow (U_1)_T$, then $h\circ \mathscr{G} := \{h \circ g \mid g \in \mathscr{G}\}$ is a bounded family of morphisms into $(U_2)_T$.
\end{itemize}
Suppose given a short exact sequence
\[
1 \longrightarrow U' \xlongrightarrow{i} U \xlongrightarrow{\pi} U'' \longrightarrow 1
\]
of smooth connected unipotent $k$-groups.
\begin{itemize}
\item[(ii)] A family $\mathscr{F}$ of $T$-morphisms $X_T \rightarrow U'_T$ is bounded in $U'$ if and only if $i \circ \mathscr{F} := \{i \circ f \mid f \in \mathscr{F}\}$ is bounded in $U$.
\item[(iii)] Suppose that $U'$ is commutative and $p$-torsion. Then, given bounded families $\mathscr{F}' \subset \Mor_T(X_T, U'_T)$ and $\mathscr{F} \subset \Mor_T(X_T, U_T)$, the family $i(\mathscr{F}')\cdot \mathscr{F}$ is bounded.
\end{itemize}
\end{lemma}

\begin{proof}
(i) Choosing $\overline{k}$-scheme isomorphisms $(U_1)_{\overline{k}} \xrightarrow{\sim} \A^{d_1}_{\overline{k}}$ and $(U_2)_{\overline{k}} \xrightarrow{\sim} \A^{d_2}_{\overline{k}}$, $h_{\overline{k}}$ goes over to a $\overline{k}$-morphism $\A^{d_1}_{\overline{k}} \rightarrow \A^{d_2}_{\overline{k}}$. This morphism sends finite-dimensional $\overline{k}$-vector subspaces of $\Gamma(X_{\overline{k}}, \calO_{X_{\overline{k}}})^{d_1}$ into finite-dimensional $\overline{k}$-vector subspaces of $\Gamma(X_{\overline{k}}, \calO_{X_{\overline{k}}})^{d_2}$.

(ii) The left-multiplication action of $U'$ on $U$ makes $U$ into a $U'$-torsor over $U''$. Because ${\rm{H}}^1(Y, \Ga) = 0$ for any affine scheme $Y$, and because $U'_{\overline{k}}$ is split unipotent, there is an isomorphism $\alpha\colon U_{\overline{k}} \xrightarrow{\sim} U'_{\overline{k}} \times U''_{\overline{k}}$ of $U'_{\overline{k}}$-torsors over $U''_{\overline{k}}$. In particular, if $(u'_0, u''_0)$ corresponds to the identity of $U_{\overline{k}}$, then $i$ corresponds to the map $u' \mapsto (u'u'_0, u''_0)$. Therefore, choosing $\overline{k}$-scheme isomorphisms $\overline{\phi'}\colon U'_{\overline{k}} \xrightarrow{\sim} \A^{d'}_{\overline{k}}$ and $\overline{\phi''}\colon U''_{\overline{k}} \xrightarrow{\sim} \A^{d''}_{\overline{k}}$, then we see by choosing the isomorphism $(\overline{\phi'} \times \overline{\phi''})\circ \alpha \colon U_{\overline{k}} \xrightarrow{\sim} \A^{d'+d''}_{\overline{k}}$ that boundedness in $U'$ is equivalent to boundedness after passing to the larger group $U$.

(iii) We argue similarly as in the proof of (ii), but replace the $\overline{k}$-scheme isomorphism $\overline{\phi'}$ with a {\em $\overline{k}$-group scheme} isomorphism $\overline{\phi'}\colon U'_{\overline{k}} \xrightarrow{\sim} (\mathbf{G}_{a,\,\overline{k}})^{d'}$ \cite[Cor.\,B.2.7]{cgp}. Then the fact that $(\overline{\phi'} \times \overline{\phi''})\circ \alpha$ is an isomorphism of $U'_{\overline{k}}$-torsors over $U''_{\overline{k}}$ implies that, via this isomorphism, the left multiplication action of $U'$ on $U$ goes over to addition in the first variable. In particular, left multiplying by elements lying in (the base change to $T$) of elements in a finite-dimensional vector space of global sections on $X$ preserves the property of lying in (the base change to $T$) of a finite-dimensional vector space, which proves (iii).
\end{proof}

We are now prepared to prove the key boundedness result that will allow us to prove that $\calMor(X, U)^+$ is representable.

%%% Proposition: Mor_{k_s}(X, U) is bounded
\begin{proposition}
\label{mor(x,U)bounded}
For a field $k$, a geometrically reduced $k$-scheme $X$ of finite type, and a wound unipotent $k$-group $U$, the family $\Mor_{k_s}(X_{k_s}, U_{k_s})$ is bounded.
\end{proposition}

\begin{proof}
We first prove the result when $U$ is commutative and $p$-torsion. Then by \cite[Prop.\,B.1.13]{cgp}, $U \simeq \ker(F\colon \Ga^n \rightarrow \Ga)$ for some $p$-polynomial $F \in k[X_1, \dots, X_n]$. By \cite[Lem.\,B.1.7]{cgp}, we may choose $F$ so that its principal part admits no nontrivial zeroes over $k$. Let $i\colon U \hookrightarrow \Ga^n$ denote the corresponding inclusion. By Lemma \ref{boundednesslargergroup}(ii), it suffices to check that the family $\{i \circ f\mid f \in \Mor_{k_s}(X_{k_s}, U_{k_s})\}$ is bounded as a family of $k_s$-morphisms $X_{k_s} \rightarrow \mathbf{G}_{a,\,k_s}^n$. Using the obvious isomorphism of $\mathbf{G}_{a,\,k_s}^n$ with $\A^n_{k_s}$, and Remark \ref{boundedfork_scriterion}, this is in turn equivalent to the statement that there is a finite-dimensional $k$-vector space $W \subset {\rm{H}}^0(X, \calO_X)$ such that $F_X^{-1}(0) \subset W^n$, where $F_X\colon {\rm{H}}^0(X, \calO_X)^n \rightarrow {\rm{H}}^0(X, \calO_X)$ is the map $(s_1, \dots, s_n) \mapsto F(s_1, \dots, s_n)$. This is the special case $V = 0$ of Proposition \ref{boundedness}.

Now we prove the general case by dimension induction on $U$, the $0$-dimensional case being trivial. So suppose that $U \neq 1$. By \cite[Prop.\,B.3.2]{cgp} and the fact that the cckp-kernel of a nontrivial smooth connected unipotent $k$-group is nontrivial (see the comment after \cite[Def.\,B.3.1]{cgp}), there is an exact sequence
\begin{equation}
\label{mor(x,U)boundedpfeqn1}
1 \longrightarrow U' \longrightarrow U \longrightarrow U'' \longrightarrow 1
\end{equation}
of wound unipotent $k$-groups such that $U' \subset U$ is nontrivial, central, and $p$-torsion. In particular, as was shown in the preceding paragraph, the proposition holds for maps into $U'$, and it holds for maps into $U''$ by induction. We must deduce that it also holds for $k_s$-maps into $U$.

Arguing as at the beginning of this proof, $U'$ is the vanishing locus of a $p$-polynomial $F \in k[X_1, \dots, X_n]$ whose principal part has no nontrivial zeroes. We then push out the exact sequence (\ref{mor(x,U)boundedpfeqn1}) along the corresponding inclusion $U' \hookrightarrow \Ga^n$ to obtain a commutative diagram of exact sequences
\[
\begin{tikzcd}
1 \arrow{r} & U' \arrow{r} \arrow[d, hookrightarrow] & U \arrow[d, hookrightarrow, "j"] \arrow{r}{\pi} & U'' \arrow[d, equals] \arrow{r} & 1 \\
1 \arrow{r} & \Ga^n \arrow[d, twoheadrightarrow, "F"] \arrow{r}{i} & W \arrow{r}{q} \arrow[d, twoheadrightarrow, "F'"] & U'' \arrow{r} & 1 \\
& \Ga \arrow[r, equals] & \Ga &&
\end{tikzcd}
\]
where $W$ is defined so that the top left square is a pushout diagram. Via the left-multiplication action, $W$ is an fppf $\Ga^n$-torsor over $U''$. Because $U''$ is affine, ${\rm{H}}^1(U'', \Ga^n) = 0$, so there is a {\em scheme-theoretic} section (we do not claim that it is a homomorphism) $s\colon U'' \rightarrow W$. Because $\Mor_{k_s}(X_{k_s}, U''_{k_s})$ is bounded, the subset $\pi \circ \Mor_{k_s}(X_{k_s}, U_{k_s}) := \{\pi \circ g \mid g \in \Mor_{k_s}(X_{k_s}, U_{k_s})\}$ also is. The section $s$ gives an isomorphism $\alpha \colon W \xrightarrow{\sim} \Ga^n \times U''$ of $\Ga^n$-torsors over $U''$. By Lemma \ref{boundednesslargergroup}(i), the family $\mathscr{F} := s\circ \pi \circ \Mor_{k_s}(X_{k_s}, U_{k_s})$ of $k_s$-morphisms $X_{k_s} \rightarrow W_{k_s}$ is bounded. The family $\mathscr{F}$ has the further property that 
\begin{equation}
\label{mor(x,U)boundedpfeqn4}
q \circ \mathscr{F} = \pi \circ \Mor_{k_s}(X_{k_s}, U_{k_s}).
\end{equation}
It follows that $\mathscr{F} \subset i(\Ga^n(X_{k_s}))\cdot j(U_{k_s}(X_{k_s}))$, whence $F' \circ \mathscr{F} \subset F(\Ga^n(X_{k_s}))$. That is, 
\begin{equation}
\label{mor(x,U)boundedpfeqn3}
F'\circ \mathscr{F} \subset F_X({\rm{H}}^0(X_{k_s}, \calO_{X_{k_s}})^n),
\end{equation}
where $F_X \colon {\rm{H}}^0(X_{k_s}, \calO_{X_{k_s}})^n \rightarrow {\rm{H}}^0(X_{k_s}, \calO_{X_{k_s}})$ is the map $(s_1, \dots, s_n) \mapsto F(s_1, \dots, s_n)$.

On the other hand, by Lemma \ref{boundednesslargergroup}(i), the family $F' \circ \mathscr{F}$ of $k_s$-morphisms $X_{k_s} \rightarrow \mathbf{G}_{a,\,k_s}$ is bounded. This means that there is a finite-dimensional $k_s$-vector space $V \subset {\rm{H}}^0(X_{k_s}, \calO_{X_{k_s}})$ such that 
\[
F' \circ \mathscr{F} \subset V.
\]
By Proposition \ref{boundedness} (and using Lemma \ref{nontrivialzeroessepble} to ensure that the principal part of $F$ has no nontrivial zeroes over $k_s$), it follows that there is a finite-dimensional $k_s$-vector space $T \subset {\rm{H}}^0(X_{k_s}, \calO_{X_{k_s}})$ such that $F_X^{-1}(F' \circ \mathscr{F}) \subset T^n$. Combining this with (\ref{mor(x,U)boundedpfeqn3}), we deduce that $F' \circ \mathscr{F} \subset F(T^n)$. Therefore, every element of $\mathscr{F}$ may be modified by an element of $i(T^n)$ to lie in $j(U_{k_s}(X_{k_s}))$. By Lemma \ref{boundednesslargergroup}(i),(iii), after so modifying each element of $\mathscr{F}$, we thereby obtain a family $\mathscr{G} \subset U_{k_s}(X_{k_s})$ such that
$j\circ \mathscr{G}$ is bounded, hence $\mathscr{G}$ is bounded by Lemma \ref{boundednesslargergroup}(ii), and, thanks to (\ref{mor(x,U)boundedpfeqn4}), $\mathscr{G}$ has the property that 
\[
U_{k_s}(X_{k_s}) \subset U'_{k_s}(X_{k_s})\cdot \mathscr{G}.
\]
Since $U'_{k_s}(X_{k_s})$ is bounded, Lemma \ref{boundednesslargergroup}(iii) then implies that $U_{k_s}(X_{k_s})$ is bounded, as desired.
\end{proof}

%%% Section: Constructing the moduli spaces: the unipotent case
\section{Constructing the moduli spaces: the unipotent case}
\label{modulispacesunipsection}

In this section we will construct the moduli spaces whose existence is asserted in Theorem \ref{maintheorem} in the case that $G = U$ is wound unipotent. Choose a $k'$-scheme isomorphism $\phi'\colon U_{k'} \xrightarrow{\sim} \A^d_{k'}$ for some finite extension field $k'/k$. Let $X$ be a $k$-scheme, and let $V \subset {\rm{H}}^0(X, \calO_X)$ be a finite-dimensional $k$-vector subspace. Consider the functor $\calMor_{{\phi'},\,V}(X, U)\colon \{\mbox{$k$-schemes}\} \rightarrow \{\mbox{sets}\}$ which sends a $k$-scheme $T$ to the set of all $T$-morphisms $f\colon X_T \rightarrow U_T$ such that ${\phi'}_T\circ f_{k'}\colon X_{T_{k'}} \rightarrow \A^d_{T_{k'}}$ lies in $({\rm{H}}^0(T_{k'}, \calO_{T_{k'}}) \otimes_k V)^d$. We first prove that $\calMor_{{\phi'},\,V}(X, U)$ is represented by a finite type $k$-scheme.

%%% Lemma: \calMor_{\overline{\phi},\,V} is represented by a finite type k-scheme
\begin{lemma}
\label{Mor_phi,V}
For any finite-dimensional $k$-vector subspace $V \subset {\rm{H}}^0(X, \calO_X)$, the subfunctor $\calMor_{\phi',\,V}(X, U) \subset \calMor(X, U)$ is represented by a $k$-scheme of finite type.
\end{lemma}

\begin{proof}
Let $\pi_i\colon {\rm{Spec}}(k' \otimes_k k') \rightarrow {\rm{Spec}}(k')$, $i = 1, 2$, denote the projection morphisms. For a $k$-scheme $T$, descent theory implies that a $T$-morphism $h\colon X_T \rightarrow U_T$ is the same thing as a $T_{k'}$-morphism $h'\colon X_{T_{k'}} \rightarrow U_{T_{k'}}$ such that $\pi_1^*g' = \pi_2^*g'\colon X_{T_{k' \otimes_k k'}} \rightarrow U_{T_{k' \otimes_k k'}}$. Using the isomorphism $\phi'$, this is the same as a $T_{k'}$-morphism $g'\colon X_{T_{k'}} \rightarrow \A^d_{T_{k'}}$ such that 
\begin{equation}
\label{Mor_phi,Vpfeqn1}
(\pi_1^*\phi')^{-1}\circ (\pi_1^*g') = (\pi_2^*\phi')^{-1}\circ(\pi_2^*g') \colon X_{T_{k' \otimes_k k'}} \rightarrow \A^d_{T_{k' \otimes_k k'}}.
\end{equation}

A morphism $g' \colon X_{T_{k'}} \rightarrow \A^d_{T_{k'}}$ is just a $d$-tuple of global sections of ${\rm{H}}^0(X_{T_{k'}}, \calO_{X_{T_{k'}}})$. Locally on $T$, this last space equals $(k' \otimes_k {\rm{H}}^0(X, \calO_X)) \otimes_k {\rm{H}}^0(T, \calO_T)$, and the requirement that the morphism $g$ to which $g'$ descends lies in $\calMor_{\phi',\,V}(X, U)$ is exactly that each entry of this $d$-tuple lies in $(k' \otimes_k V) \otimes_k {\rm{H}}^0(T, \calO_T)$. The functor of such $d$-tuples is the vector group associated to the finite-dimensional $k$-vector space $k' \otimes_k V$. We must check that the condition (\ref{Mor_phi,Vpfeqn1}) defines a closed subscheme of this vector group.

In fact, choosing $k$-bases $\{e_i\}_{i \in I}$ and $\{f_j\}_{j \in J}$ for $k'$ and $V$, respectively, so that the coordinates on the vector group are $\mathbf{t} := (t_{i,\,j})_{i \in I, j \in J}$, the equation (\ref{Mor_phi,Vpfeqn1}) amounts (after composing both sides with $\pi_2^*\phi'$) on each of the $d$-coordinates of $\A^d_{k' \otimes k'}$ to an equation of the form
\begin{equation}
\label{Mor_phi,Vpfeqn2}
\sum_{r=1}^N F_r(\mathbf{t}) G_r = 0,
\end{equation}
for some {\em fixed} functions $G_r \in {\rm{H}}^0(X, \calO_X)$ and some {\em fixed} polynomials $F_r \in (k' \otimes_k k')[t_{i,\,j}\mid i \in I, j \in J]$. We must show that such an equation translates to some polynomial equations over $k$ in the $t_{i,\,j}$. 

First we note that, for any $k$-scheme $S$, the ${\rm{H}}^0(S, \calO_S)$-linear relations among the $G_r$ are obtained via tensor product from the $k$-relations among the $G_r$. Indeed, this follows from the fact that, locally on $S$, ${\rm{H}}^0(X_S, \calO_{X_s}) = {\rm{H}}^0(S, \calO_S) \otimes_k {\rm{H}}^0(X, \calO_X)$. Applying this with $S := T_{k' \otimes_k k'}$, we find that (\ref{Mor_phi,Vpfeqn2}) is equivalent to a collection of polynomial equations in the $t_{i,\,j}$ with coefficients in $k' \otimes_k k'$. Working coordinate by coordinate with respect to a $k$-basis for $k' \otimes_k k'$, we then find that such equations are in turn equivalent to a collection of polynomial equations over $k$ in the $t_{i,\,j}$. This completes the proof of the lemma.
\end{proof}

We will also require the following lemma relating the schemes $\calMor_{\phi',\,V}(X, U)$ as $V$ varies. Note that, for $V \subset W \subset {\rm{H}}^0(X, \calO_X)$ finite-dimensional, there is a natural inclusion $\calMor_{\phi',\,V}(X, U) \hookrightarrow \calMor_{\phi',\,W}(X, U)$.

%%% Lemma: Mor_phi, V is a closed subscheme of Mor_phi, W
\begin{lemma}
\label{MorphiVclosed}
For finite-dimensional $k$-vector subspaces $V \subset W \subset {\rm{H}}^0(X, \calO_X)$, the natural inclusion $\calMor_{\phi',\,V}(X, U) \hookrightarrow \calMor_{\phi',\,W}(X, U)$ is a closed embedding.
\end{lemma}

\begin{proof}
Let $T$ be a $k$-scheme, and consider a $T$-point $g \colon X_T \rightarrow U_T$ of $\calMor_{\phi',\,W}(X, U)$. We must show that the locus where $T$ maps into $\calMor_{\phi',\,V}(X, U)$ is a closed subscheme of $T$ (more precisely, there is a closed subscheme $S$ of $T$ such that a $T$-scheme $T'$ factors through $S$ precisely when $g|_{T'} \in \calMor_{\phi',\,V}(X, U)$). It suffices to check this after scalar extension to $k'$. Choosing a $k$-basis for $V$ and completing it to a $k$-basis for $W$ by adding the elements $\{e_i\}_{i \in I}$, the condition that a morphism factor through $\calMor_{\phi',\,V}(X, U)$ is that the coefficients of the $e_i$ vanish. This defines a closed subscheme of $T$.
\end{proof}

We are now ready to construct the moduli spaces $\calMor(X, U)^+$ for $U$ wound unipotent. Recall the notion of the maximal geometrically reduced closed subscheme of a locally finite type $k$-scheme \cite[Lem.\,C.4.1]{cgp}.

%%% Theorem: Mor(X, U)^+ when U wound unipotent
\begin{theorem}
\label{Mor(X,U)^+}
For a field $k$, a geometrically reduced $k$-scheme $X$ of finite type, and a wound unipotent $k$-group scheme $U$, there is a unique subfunctor $\calMor(X, U)^+ \subset \calMor(X, U)$ with the following two properties:
\begin{itemize}
\item[(i)] The inclusion $\calMor(X, U)^+ \subset \calMor(X, U)$ induces an equality on $T$-points for every geometrically reduced $k$-scheme $T$.
\item[(ii)] The functor $\calMor(X, U)^+$ is represented by a smooth $k$-group scheme.
\end{itemize}
Furthermore, the scheme in (ii), which we also denote $\calMor(X, U)^+$, is a finite type unipotent $k$-group scheme with wound identity component.
\end{theorem}

The statement that a finite type smooth $k$-group scheme is unipotent, in characteristic $p$, amounts to the statement that its identity component is unipotent, and its \'etale component group is of $p$-power order.

\begin{proof}
Choose a $k'$-scheme isomorphism $\phi'\colon U_{k'} \xrightarrow{\sim} \A^d_{k'}$ for some finite extension field $k'/k$. By Proposition \ref{mor(x,U)bounded}, there is a finite-dimensional $k$-vector space $V \subset {\rm{H}}^0(X, \calO_X)$ such that for every $k_s$-morphism $X_{k_s} \rightarrow U_{k_s}$, upon extending scalars via $k'$ and using the isomorphism $\phi'$, the associated $d$-tuple of global sections of $X_{k_s \otimes_k k'}$ all lie in $(k_s \otimes_k k') \otimes_k V$. That is, every $k_s$-morphism $X_{k_s} \rightarrow U_{k_s}$ defines a $k_s$-point of $\calMor_{\phi',\,V}(X, U)$. 

We now claim that, for any geometrically reduced $k$-scheme $T$, a $T$-morphism $X_T \rightarrow U_T$ defines a $T$-point of $\calMor_{\phi',\,V}(X, U) \subset \calMor(X, U)$. It suffices to treat the case of affine $T$, and because $U$ is of finite type, by writing any geometrically reduced $k$-algebra as a filtered direct limit of geometrically reduced $k$-algebras of finite type, we may assume that $T$ is of finite type. Then a $T$-morphism $X_T \rightarrow U_T$ defines a $T$-point of $\calMor_{\phi',\,W}(X, U)$ for some finite-dimensional $W \subset {\rm{H}}^0(X, \calO_X)$, and we may enlarge $W$ in order to assume that $V \subset W$. The $k_s$-points of $T$ factor through the subscheme $i \colon \calMor_{\phi',\,V}(X, U) \rightarrow \calMor_{\phi',\,W}(X, U)$. Since $T$ is geometrically reduced, $T(k_s)$ is Zariski dense in $T$. Because $i$ is a closed embedding by Lemma \ref{MorphiVclosed}, it follows that $T$ itself factors through $\calMor_{\phi',\,V}(X, U)$, as claimed.

We therefore see that the inclusion $\calMor_{\phi',\,V}(X, U) \subset \calMor(X, U)$ induces an equality on the category of geometrically reduced $k$-schemes. Because maps from geometrically reduced $k$-schemes of finite type into $\calMor_{\phi',\,V}(X, U)$ factor uniquely through its maximal geometrically reduced closed subscheme $\calMor_{\phi',\,V}(X, U)^\sharp \subset \calMor(X, U)$, we see that this finite type, geometrically reduced $k$-scheme agrees with the functor $\calMor(X, U)$ on geometrically reduced $k$-schemes. It follows that it is a $k$-group scheme, hence smooth, because it is geometrically reduced.

To show that $\calMor(X, U)^+$ is unipotent, we note that it is $p^n$-torsion for some $n > 0$, because $U$ is. It only remains to show that its identity component is wound. To show this, suppose given a $k$-scheme morphism $f\colon \A^1_k \rightarrow \calMor(X, U)^+$ such that $f(0) = 1$. We must show that $f = 1$. By the very definition of $\calMor(X, U)^+$, such an $f$ is the same as a $k$-morphism $g\colon X \times \A^1 \rightarrow U$ such that $g|_{X \times 0} = 1$. For each $x \in X(k_s)$, $g|_{x \times \A^1}$ defines a $k_s$-map $\A^1_{k_s} \rightarrow U_{k_s}$ sending $0$ to $1$. Because $U_{k_s}$ is wound, it follows that $g|_{x \times \A^1} = 1$. Since $X$ is geometrically reduced, $X(k_s)$ is Zariski dense in $X$, so it follows that $g = 1$.
\end{proof}

%%% Section: Constructing the moduli spaces in general
\section{Constructing the moduli spaces in general}
\label{constructiongmoduligeneralsection}

In this section, we construct the schemes $\calMor(X, G)^+$ in general (that is, beyond the case in which $G$ is wound unipotent, which was treated in the previous section). For technical reasons, it will be convenient to first deal with moduli spaces of {\em pointed} morphisms.

\begin{definition}
For a pair of pointed $k$-schemes $(X, x)$ and $(Y, y)$ (that is, $X$ and $Y$ are $k$-schemes, and $x \in X(k), y \in Y(k)$), the functor
\[
\calMor((X, x), (Y, y))\colon \{\mbox{$k$-schemes}\} \rightarrow \{\mbox{sets}\}
\]
is defined by the formula $T \mapsto \Mor_T((X_T, x_T), (Y_T, y_T))$ -- that is, $T$ is sent to the set of $T$-morphisms $f\colon X_T \rightarrow Y_T$ such that $f(x_T) = y_T$.
\end{definition}

We begin by treating pointed morphisms into semiabelian varieties. Of particular importance, from a purely technical standpoint, will be the fact that these moduli spaces are \'etale.

%%% Theorem: Pointed morphisms into semiabelian varieties
\begin{theorem}
\label{pointedmorsemiabelianvar}
For a semiabelian variety $S$ over a field $k$, and a connected and geometrically reduced $k$-scheme $X$ of finite type equipped with a point $x \in X(k)$, there is a unique subfunctor $\calMor((X, x), (S, 0))^+ \subset \calMor((X, x), (S, 0))$ with the following two properties:
\begin{itemize}
\item[(i)] The inclusion $\calMor((X, x), (S, 0))^+ \subset \calMor((X, x), (S, 0))$ is an equality on $T$-points for every geometrically reduced $k$-scheme $T$.
\item[(ii)] The functor $\calMor((X, x), (S, 0))^+$ is represented by a smooth $k$-group scheme.
\end{itemize}
Furthermore, the scheme described in (ii), which we also denote $\calMor((X, x), (S, 0))^+$, is an \'etale $k$-group scheme with finitely-generated group of $k_s$-points.
\end{theorem}

\begin{proof}
We first note that, because $X$ is connected and equipped with a $k$-point, it is geometrically connected. Further, we are free by Galois descent to prove the theorem over $k_s$, so we may assume that $k$ is separably closed. In that case, the theorem asserts in particular that $\calMor((X, x), (S, 0))^+$ is constant. If we construct a constant $k$-scheme representing the restriction of the functor $\calMor((X, x), (S, 0))$ to the category of all {\em geometrically reduced} $k$-schemes, then -- because it is constant -- it is clearly a subfunctor of the functor $\calMor((X, x), (S, 0))$ on the category of {\em all} $k$-schemes. We therefore show that this restricted functor is represented by a constant $k$-group scheme with finitely-generated group of $k$-points.

If we construct a $k$-scheme representing the functor $\calMor((X, x), (S, 0))$ on the category of {\em affine} geometrically reduced $k$-schemes, then it also does so on the category of all geometrically reduced $k$-schemes, because this functor is a Zariski sheaf. Since both the functor $\calMor((X, x), (S, 0))$ restricted to the category of geometrically reduced $k$-schemes and the functor represented by an \'etale (or even a locally finite type) $k$-scheme commute with filtered direct limits of rings, it suffices to prove that the restriction of the functor $\calMor((X, x), (S, 0))$ to the category of geometrically reduced $k$-schemes {\em of finite type} is represented by a constant group scheme with finitely-generated group of $k$-points.

The statement that the restriction of the functor $\calMor((X, x), (S, 0))$ to the category of all geometrically reduced $k$-schemes of finite type is represented by a constant $k$-scheme is equivalent to the statement that, for any nonempty, geometrically connected, and geometrically reduced $k$-scheme $T$ of finite type, and any $k$-morphism $f\colon X \times T \rightarrow S$ such that $f|_{x \times T} = 0$, there is a $k$-morphism $g\colon X \rightarrow S$ such that $f = g\circ \pi_X$, where $\pi_X\colon X \times T \rightarrow X$ is the projection. In order to prove this, in turn, we are free to extend scalars to $\overline{k}$ and thereby assume (for the sake of proving this claim) that $k$ is algebraically closed. Choosing a point $t \in T(k)$ and modifying $f$ by $f|_{X \times t}$, we may assume that $f|_{X \times t} = 0$, and we wish to show that $f = 0$. 

First consider the case in which $X$ and $T$ are (geometrically) integral. Because $S$ is separated, we are free to shrink $X$ and $T$ and thereby assume that they are quasi-projective. The result then follows from \cite[Th.\,2]{rosenlicht}. Now consider the case in which $T$ is integral, but $X$ is only reduced and connected. Let $X_1, \dots, X_n$ be the irreducible components of $X$, and suppose that $x \in X_{i_0}$. Then because $X_{i_0}$ is integral, we deduce that $f|_{X_{i_0} \times T} = 0$. For any component $X_j$ intersecting $X_{i_0}$, let $x' \in X_j \cap X_{i_0}$ be a point of intersection. Then $f|_{x' \times T} = 0$, so once again by the integral case of the lemma, we find that $f|_{X_j \times T} = 0$. Continuing in this manner, and because $X$ is connected, we find that $f$ vanishes on all of $X \times T$. The deduction of the general case from the case when $T$ is integral proceeds via a similar argument. This completes the proof that $\calMor((X, x), (S, 0))^+$ is represented by an \'etale $k$-scheme.

The statement that the constant $k$-scheme $\calMor((X, x), (S, 0))^+$ has finitely-generated group of $k$-points simply says that the group $\Mor((X, x), (S, 0))$ is finitely-generated, and this may be checked over $\overline{k}$, so we may once again assume that $k$ is algebraically closed. The desired assertion is equivalent to the statement that the group of $k$-morphisms $X \rightarrow S$ is finitely-generated modulo constant maps.
Let $X_1, \dots, X_n$ denote the irreducible components of $X$. Then we have an inclusion
\[
\calMor(X, S) \hookrightarrow \prod_{i=1}^n \calMor(X_i, S),
\]
and this map remains injective when $\calMor(X, S)$ and each $\calMor(X_i, S)$ are replaced by their quotients modulo constants, because $X$ is connected. We thereby reduce to the case in which $X$ is (geometrically) integral. Because $S$ is separated, we are free to shrink $X$ and thereby assume that it is quasi-projective. The result then follows from \cite[Th.\,1]{rosenlicht}.
\end{proof}

Similar to the often-used smooth-\'etale site, we now define the geometrically reduced \'etale site of a field.

%%% Definition: Geometrically reduced \'etale site\
\begin{definition}
\label{geometricallyreducedetalesite}
Let $k$ be a field. The {\em geometrically reduced \'etale site} of $k$ is the site whose underlying category is the category of all geometrically reduced $k$-schemes, and whose covers are \'etale covers.
\end{definition}

The following two technical lemmas will prove useful.

%%% Lemma: Short exact sequence of sheaves on geom reduced etale site with etale cokernel is exact
\begin{lemma}
\label{sesgeometalesite}
Given a sequence
\[
1 \longrightarrow G' \longrightarrow G \xlongrightarrow{f} E
\]
of smooth group schemes over a field $k$, with $E$ \'etale, if the sequence is exact as a sequence of sheaves on the geometrically reduced \'etale site of $k$, then it is exact as a sequence of $k$-group schemes.
\end{lemma}

\begin{proof}
The exactness as a sequence of sheaves on the geometrically reduced \'etale site implies that the map $\psi \colon G' \rightarrow \ker(f)$ induces isomorphisms on $T$-points for all geometrically reduced $k$-schemes $T$. This would imply that $\psi$ is an isomorphism of $k$-group schemes (rather than just of sheaves on the geometrically reduced \'etale site) provided that $\ker(f)$ was smooth. So it suffices to prove this smoothness. In fact, the \'etaleness of $E$ implies that $G^0 \subset \ker(f)$, hence $\ker(f)^0 = G^0$ is smooth, therefore so is $\ker(f)$.
\end{proof}

%%% Lemma: Codomain a constant group scheme implies representable
\begin{lemma}
\label{codomainconstant}
Given an exact sequence
\[
1 \longrightarrow G \longrightarrow \mathscr{F} \xlongrightarrow{\phi} E
\]
of group sheaves on the geometrically reduced \'etale site of a field $k$, with $G$ a smooth $k$-group scheme with affine identity component and $E$ an \'etale $k$-group scheme, then $\im(\phi)$ is an \'etale $k$-group scheme, $\mathscr{F}$ is represented by a smooth $k$-group scheme, and the induced sequence
 \[
1 \longrightarrow G \longrightarrow \mathscr{F} \xlongrightarrow{h} \im(\phi) \longrightarrow 1
\]
is a short-exact sequence of $k$-group schemes.
\end{lemma}

\begin{proof}
We may extend scalars and thereby assume that $k = k_s$. We may also replace $G$ by $G^0$ and $E$ by the extension of $E$ by $G/G^0$ and thereby assume that $G$ itself is affine. Due to the effectivity of fpqc descent for relatively affine schemes and Lemma \ref{sesgeometalesite}, the only things that we have to prove are that the image of the map $\mathscr{F} \rightarrow E$ is a constant $k$-subgroup scheme of $E$ and that $h$ is surjective as a map of $k$-group schemes. In fact, any subsheaf $\mathscr{G}$ of $E$ on the geometrically reduced \'etale site of $k$ is the constant $k$-group scheme associated to a subgroup of $E(k)$. This follows from the fact that, for $1 \neq e \in E(k)$, if $e \in \mathscr{G}(T)$ for some geometrically reduced finite type $T/k$, then $T(k_s) \neq \emptyset$, so pulling back $e$ along a $k'$-point for some finite separable field extension $k'/k$ implies that $e \in \mathscr{G}(k')$, and then the sheaf property implies that $e \in \mathscr{G}(k)$. For the surjectivity of $h$ as a map of $k$-group schemes (that is, as a map of fppf sheaves), due to the smoothness of $\im(\phi)$ it suffices to prove surjectivity on $k$-points, which holds because $k = k_s$ and $h$ is surjective as a map of sheaves on the geometrically reduced \'etale site.
\end{proof}

We shall have occasion to use the following definition.

\begin{definition}
\label{almostunipotentdef}
We say that a finite type group scheme $G$ over a field $k$ is {\em almost unipotent} if the smooth connected $\overline{k}$-group $(G_{\overline{k}})_{\rm{red}}^0$ is unipotent.
\end{definition}

Almost unipotent groups enjoys the expected permanence properties.

%%% Proposition: Permanence properties of almost unipotence
\begin{proposition}
\label{almostunippermanence}
Let $k$ be a field. Then a finite type $k$-group scheme $G$ is almost unipotent if and only if it is affine and $G_{\overline{k}}$ contains no nontrivial $\overline{k}$-torus. Furthermore,
\begin{itemize}
\item[(i)] $($closure under subgroups$)$ a closed $k$-subgroup scheme of an almost unipotent $k$-group is almost unipotent;
\item[(ii)] $($closure under quotients$)$ given a $k$-homomorphism $G \rightarrow H$ of finite type $k$-groups that is surjective as a map of topological spaces, if $G$ is almost unipotent, then $H$ is also almost unipotent;
\item[(iii)] $($closure under extensions$)$ given an exact sequence
\[
1 \longrightarrow G' \longrightarrow G \longrightarrow G'' \longrightarrow 1
\]
of finite type $k$-groups with $G'$ and $G''$ almost unipotent, then $G$ is also almost unipotent.
\end{itemize}
\end{proposition}

\begin{proof}
First we prove the first assertion. The group $G$ is affine if and only if $G_{\overline{k}}$ is (e.g., by the cohomological criterion for affineness), so we may assume that $k = \overline{k}$. Furthermore, $G$ is affine if and only if $G_{\rm{red}}^0$ is. Finally, the $k = \overline{k}$-tori of $G$ are all contained in $G_{\rm{red}}^0$. To prove the first assertion, we are therefore free to assume that $G$ is smooth and connected. A unipotent group is affine and contains no nontrivial torus. Conversely, if the smooth connected $G$ is affine and contains no nontrivial tori, then it is unipotent \cite[Ch.\,IV, Cor.\,11.5(2)]{borelalggroups}. This proves the first assertion of the proposition. It remains to prove (i)-(iii). The assertions (i) and (iii) follow immediately from the first assertion, together (in the case of (iii)) with the fact that an extension of affine group schemes is affine. For (ii), the map $G \rightarrow H$ induces a surjective map on $\overline{k}$-points, hence a surjection of smooth finite type $\overline{k}$-group schemes $(G_{\overline{k}})_{\rm{red}} \twoheadrightarrow (H_{\overline{k}})_{\rm{red}}$. It follows that we also obtain a surjection on identity components. The unipotence of $(G_{\overline{k}})_{\rm{red}}^0$ therefore implies that of $(H_{\overline{k}})_{\rm{red}}^0$.
\end{proof}

%%% Proposition: Solvable group schemes are extensions of semiabelian varieties by almost-unipotent groups
\begin{proposition}
\label{extsemiabalmunip}
If $G$ is a smooth connected solvable group scheme over a field $k$, then $G$ may be written as an extension
\[
1 \longrightarrow G' \longrightarrow G \longrightarrow S \longrightarrow 1
\]
of finite type $k$-group schemes with $G'$ almost unipotent and $S$ a semiabelian variety.
\end{proposition}

In fact, Proposition \ref{extsemiabalmunip} holds without the assumption that $G$ is smooth and connected, but we do not require this greater generality.

\begin{proof}
By Chevalley's Theorem, $G_{k_{\perf}}$ is an extension of an abelian variety by a linear algebraic group, hence the same holds over $G_{k^{1/p^n}}$ for some $n > 0$. Via the field extension $k \rightarrow k$, $x \mapsto x^{p^n}$, the $k$-isomorphism $k \rightarrow k^{1/p^n}$, $x \mapsto x^{1/p^n}$, identifies the $k$-group $G^{(p^n)}$ with $G_{k^{1/p^n}}$, so it follows that $G^{(p^n)}$ is an extension {\em over $k$} of an abelian variety by a linear algebraic $k$-group. The $n$-fold Frobenius morphism $G \rightarrow G^{(p^n)}$ is a surjective $k$-homomorphism with finite kernel. By Proposition \ref{almostunippermanence}(iii), therefore, it suffices to prove the proposition for $G^{(p^n)}$, so we may assume that $G$ is an extension of an abelian variety by a linear algebraic group, and we may then assume that $G$ is linear algebraic, provided we prove that $G'$ may be chosen to be a characteristic $k$-subgroup of $G$ (to ensure normality in the larger group which is an extension of an abelian variety by $G$). Over $k_{\perf}$, $G$ is an extension of a torus by a smooth connected unipotent group, so once again this holds for some $G^{(p^m)}$. Therefore, because $\ker(F_{G/k}^{(p^m)})$ and the unipotent radical of $G^{(p^m)}$ are characteristic subgroups, we are done.
\end{proof}

We now prove the existence of the moduli spaces $\calMor((X, x), (G, 1))^+$ of {\em pointed morphisms}.

%%% Theorem: Moduli spaces of pointed morphisms
\begin{theorem}
\label{modulipointedmorphisms}
For a field $k$, a connected and geometrically reduced finite type $k$-scheme $X$ equipped with a $k$-point $x \in X(k)$, and a solvable, finite type $k$-group scheme $G$ not containing a $k$-subgroup scheme $k$-isomorphic to $\Ga$, there is a unique subfunctor $\calMor((X, x), (G, 1))^+ \subset \calMor((X, x), (G, 1))$ with the following two properties:
\begin{itemize}
\item[(i)] The inclusion $\calMor((X, x), (G, 1))^+ \subset \calMor((X, x), (G, 1))$ induces an equality on $T$-points for every geometrically reduced $k$-scheme $T$.
\item[(ii)] The functor $\calMor((X, x), (G, 1))^+$ is represented by a smooth $k$-group scheme.
\end{itemize}
Furthermore, the scheme in (ii), which we also denote $\calMor((X, x), (G, 1))^+$, has wound unipotent identity component, and its \'etale component group has finitely-generated group of $k_s$-points.
\end{theorem}

\begin{proof}
This result is true when $G$ is a semiabelian variety, by Theorem \ref{pointedmorsemiabelianvar}. Next we prove it when $G = U$ is wound unipotent. In this case, the analogue of the theorem for unpointed morphisms holds by Theorem \ref{Mor(X,U)^+}. Consider the $k$-group homomorphism $\calMor(X, U)^+ \rightarrow U$ which sends a morphism $f$ to $f(x)$. The kernel $W$ of this homomorphism is a finite type unipotent $k$-group scheme admitting no copy of $\Ga$ and such that the functor $h_W$ defined by $W$ agrees with $\calMor((X, x), (U, 1))$ on geometrically reduced $k$-schemes. The maximal smooth $k$-subgroup scheme of $W$ is then a subfunctor of $\calMor((X, x), (U, 1))$ which also agrees with it on geometrically reduced $k$-schemes.

We now turn to the general case. First note that $X$ is geometrically connected, because it is connected and contains a $k$-point. In conjunction with Galois descent, this allows us to replace $k$ by $k_s$ and assume that $k$ is separably closed. Consider the restriction $\calMor((X, x), (G, 1))_{{\rm{gr}}}$ of the functor $\calMor((X, x), (G, 1))$ to the category of geometrically reduced $k$-schemes. We will first prove that $\calMor((X, x), (G, 1))_{\rm{gr}}$ is representable by a smooth group scheme with wound identity component and finitely-generated component group. In order to show this, it suffices to show that such representability holds on the category of affine geometrically reduced $k$-schemes, and by writing arbitrary geometrically reduced $k$-algebras as filtered direct limits of finite type ones, on the category of geometrically reduced finite type $k$-schemes. We therefore prove the desired representability for the restriction of $\calMor((X, x), (G, 1))_{{\rm{gr}}}$ to the category of geometrically reduced $k$-schemes of finite type.

We claim that 
\begin{equation}
\label{modulipointedmorphismspfeqn1}
\calMor((X, x), (G, 1))_{\rm{gr}} = \calMor((X, x), (G^{\rm{sm}}, 1))_{\rm{gr}},
\end{equation}
where $G^{\rm{sm}} \subset G$ is the maximal smooth $k$-subgroup scheme. Indeed, this follows from the fact that a map from a finite type geometrically reduced $k$-scheme into $G$ factors through $G^{\rm{sm}}$ \cite[Lem.\,C.4.1, Rem.\,C.4.2]{cgp}. We may therefore replace $G$ by $G^{\rm{sm}}$. We also claim that
\[
\calMor((X, x), (G^{\rm{sm}}, 1))^+ = \calMor((X, x), ((G^{\rm{sm}})^0, 1))^+.
\]
Indeed, it suffices to show that the two functors agree on $k$-points, since any geometrically reduced $k$-scheme of finite type has Zariski dense set of $k$-points (because $k = k_s$), and for $k$-points, the assertion follows from the fact that $X$ is connected. In proving the theorem for a given $G$, we are therefore free to replace $G$ by $(G^{\rm{sm}})^0$. We may therefore assume that $G$ is smooth and connected.

By Proposition \ref{extsemiabalmunip}, we have an exact sequence
\begin{equation}
\label{modulipointedmorphismspfeqn7}
1 \longrightarrow G' \longrightarrow G \longrightarrow S \longrightarrow 1
\end{equation}
with $G'$ almost unipotent and $S$ a semiabelian variety. This induces an exact sequence of sheaves on the geometrically reduced \'etale site of $k$
\[
1 \longrightarrow \calMor((X, x), (G', 1))_{\rm{gr}} \longrightarrow \calMor((X, x), (G, 1))_{\rm{gr}} \longrightarrow \calMor((X, x), (S, 0))_{\rm{gr}}.
\]
By Theorem \ref{pointedmorsemiabelianvar}, $\calMor((X, x), (S, 0))_{\rm{gr}}$ is a constant $k$-group scheme with finitely-generated group of $k$-points. Lemma \ref{codomainconstant} therefore implies that, in order to prove the representability of $\calMor((X, x), (G, 1))_{\rm{gr}}$ by a smooth group of the desired sort (wound identity component and finitely-generated component group), it suffices to prove that of $\calMor((X, x), (G', 1))_{\rm{gr}}$. As we have seen above, it suffices in turn to replace $G$ by $((G')^{\rm{sm}})^0$, so we may assume (for the purpose of proving the representability of $\calMor((X, x), (G, 1))_{\rm{gr}}$ by a scheme of the desired shape) that $G = U$ is a wound unipotent $k$-group, which case has already been treated.

We therefore have a smooth $k$-group scheme $\calMor((X, x), (G, 1))^+$ which represents the functor $\calMor((X, x), (G, 1))_{\rm{gr}}$. The universal morphism $$(X \times \calMor((X, x), (G, 1))^+, x) \rightarrow (G \times \calMor((X, x), (G, 1))^+, 1)$$ then defines a morphism of functors $\calMor((X, x), (G, 1))^+ \rightarrow \calMor((X, x), (G, 1))$, and we wish to show that this morphism is an inclusion. As we saw previously, $\calMor((X, x), (G, 1))^+$ is unchanged upon replacing $G$ by $(G^{\rm{sm}})^0$, so, thanks to the inclusion $$\calMor((X, x), ((G^{\rm{sm}})^0, 1)) \hookrightarrow \calMor((X, x), (G, 1)),$$ we see that the assertion holds if $(G^{\rm{sm}})^0$ is wound unipotent and that, in order to prove it for general $G$, we are free to replace $G$ by $(G^{\rm{sm}})^0$ and assume that $G$ is smooth and connected. 

By Proposition \ref{extsemiabalmunip}, we have an exact sequence as in (\ref{modulipointedmorphismspfeqn7}) with $G'$ almost unipotent and $S$ a semiabelian variety. Applying Theorem \ref{pointedmorsemiabelianvar} and Lemma \ref{sesgeometalesite}, we obtain the following exact commutative diagram of group functors on the category of all $k$-schemes:
\[
\begin{tikzcd}
1 \arrow{r} & \calMor((X, x), (G', 1))^+ \arrow{r} \arrow[d, hookrightarrow] & \calMor((X, x), (G, 1))^+ \arrow{r} \arrow{d} & \calMor((X, x), (S, 0))^+ \arrow[d, hookrightarrow] \\
1 \arrow{r} & \calMor((X, x), (G', 1)) \arrow{r} & \calMor((X, x), (G, 1)) \arrow{r} & \calMor((X, x), (S, 0))
\end{tikzcd}
\]
Because the left and right vertical arrows are inclusions, it follows that the same holds for the middle arrow. This completes the proof of the theorem.
\end{proof}

We are now prepared to prove Theorem \ref{maintheorem}, whose statement we recall.

%%% Theorem: Space of morphisms representable
\begin{theorem}[Theorem $\ref{maintheorem}$]
\label{maintheoremtake2}
For a field $k$, a geometrically reduced $k$-scheme $X$ of finite type, and a solvable $k$-group scheme $G$ of finite type not containing a $k$-subgroup scheme $k$-isomorphic to $\Ga$, let $$\calMor(X, G)\colon \{\mbox{$k$-schemes}\} \rightarrow \{\mbox{groups}\}$$ denote the functor defined by the formula $T \mapsto \Mor_T(X_T, G_T)$. Then there is a unique subfunctor $\calMor(X, G)^+ \subset \calMor(X, G)$ with the following two properties:
\begin{itemize}
\item[(i)] The inclusion $\calMor(X, G)^+(T) \subset \calMor(X, G)(T)$ is an equality for all geometrically reduced $k$-schemes $T$.
\item[(ii)] The functor $\calMor(X, G)^+$ is represented by a smooth $k$-group scheme.
\end{itemize}
Furthermore, letting $\calMor(X, G)^+$ also denote the scheme in (ii), then the \'etale component group of $\calMor(X, G)^+$ has finitely-generated group of $k_s$-points. If $G = U$ is wound unipotent, then $\calMor(X, U)^+$ is a smooth, finite type, unipotent $k$-group scheme with wound identity component.
\end{theorem}

\begin{proof}
The case when $G = U$ is wound unipotent is Theorem \ref{Mor(X,U)^+}. For the general case, we may by Galois descent assume that $k = k_s$. Let $X_1, \dots, X_n$ be the connected components of $X$. Then $\calMor(X, G) = \prod_{i=1}^n \calMor(X_i, G)$, so we may assume that $X$ is (geometrically, because $k = k_s$) connected. The theorem is trivial if $X = \emptyset$, so we may assume that $X \neq \emptyset$. Because $k = k_s$ and $X$ is geometrically reduced, $X(k) \neq \emptyset$. Choose some $x \in X(k)$. Consider the moduli space $\calMor((X, x), (G, 1))^+$ constructed in Theorem \ref{modulipointedmorphisms}, and let $\calMor(X, G)_{\rm{gr}}$ denote the restriction of $\calMor(X, G)$ to the category of geometrically reduced $k$-schemes. The exact sequence
\[
1 \longrightarrow \calMor((X, x), (G, 1))^+ \longrightarrow \calMor(X, G)_{\rm{gr}} \xlongrightarrow{f \mapsto f(x)} G \longrightarrow 1 
\]
of sheaves on the geometrically reduced \'etale site of $k$ splits: the map sending $g$ to the constant map $X \rightarrow g$ yields a section $G \rightarrow \calMor(X, G)^+$. We therefore have an equality of functors on the category of geometrically reduced $k$-schemes
\[
\calMor(X, G)_{\rm{gr}} = G \ltimes \calMor((X, x), (G, 1))^+.
\]
Furthermore, the scheme on the right is a subfunctor of $\calMor(X, G)$ because the scheme $\calMor((X, x), (G, 1))^+$ is a subfunctor of $\calMor((X, x), (G, 1))$. We therefore see that the scheme $\calMor(X, G)^+ := (G \ltimes \calMor((X, x), (G, 1))^+)^{{\rm{sm}}}$ is a subfunctor of $\calMor(X, G)$ which agrees with $\calMor(X, G)$ on the subcategory of geometrically reduced $k$-schemes. Furthermore, because the component group of $\calMor((X, x), (G, 1))^+$ has finitely-generated group of $k$-points, the same holds for $\calMor(X, G)^+$.
\end{proof}

For completeness, we also construct the moduli spaces $\calMor((X, x), (G, 1))^+$ described in Theorem \ref{modulipointedmorphisms}, but this time without connectedness assumptions on $X$.

%%% Theorem: Moduli spaces of pointed morphisms, general case
\begin{theorem}
\label{modulipointedmorphismsgeneral}
For a field $k$, a geometrically reduced finite type $k$-scheme $X$ equipped with a $k$-point $x \in X(k)$, and a solvable, finite type $k$-group scheme $G$ not containing a $k$-subgroup scheme $k$-isomorphic to $\Ga$, there is a unique subfunctor $\calMor((X, x), (G, 1))^+ \subset \calMor((X, x), (G, 1))$ with the following two properties:
\begin{itemize}
\item[(i)] The inclusion $\calMor((X, x), (G, 1))^+ \subset \calMor((X, x), (G, 1))$ induces an equality on $T$-points for every geometrically reduced $k$-scheme $T$.
\item[(ii)] The functor $\calMor((X, x), (G, 1))^+$ is represented by a smooth $k$-group scheme.
\end{itemize}
Furthermore, the scheme in (ii), which we also denote $\calMor((X, x), (G, 1))^+$, has the property that its \'etale component group has finitely-generated group of $k_s$-points.
\end{theorem}

\begin{proof}
Let $W$ denote the kernel of the $k$-group scheme homomorphism $$\calMor(X, G)^+ \subset \calMor(X, G) \rightarrow G$$ sending $f$ to $f(x)$. Then $\calMor((X, x), (G, 1))^+$ is the maximal smooth $k$-subgroup scheme of $W$.
\end{proof}

%%% Moduli spaces of homomorphisms and geometric class field theory over a base
\section{Moduli spaces of homomorphisms and geometric class field theory over a base}
\label{moduliofhomssection}

As an application of Theorem \ref{maintheoremtake2}, we will now construct moduli spaces of homomorphisms between group schemes and use them to show that one of the major results of geometric class field theory, Theorem \ref{albanese}, remains true after scalar extension to a geometrically reduced $k$-scheme, provided that the group $G$ does not contain a copy of $\Ga$ (for example, if $G$ is wound unipotent). We first construct the moduli spaces of homomorphisms.

%%% Theorem: Moduli spaces of homomorphisms
\begin{theorem}
\label{modulispacesofhoms}
Let $G$ and $H$ be finite type $k$-group schemes over a field $k$, with $H$ smooth and $G$ solvable and not containing a $k$-subgroup scheme $k$-isomorphic to $\Ga$. Consider the functor
\[
\calHom(H, G)\colon \{\mbox{$k$-schemes}\} \rightarrow \{\mbox{sets}\}
\]
which sends $T$ to $\Hom_{\mbox{$T$-gps}}(H_T, G_T)$. Then there is a unique subfunctor $\calHom(H, G)^+ \subset \calHom(H, G)$ with the following two properties:
\begin{itemize}
\item[(i)] The inclusion $\calHom(H, G)^+ \subset \calHom(H, G)$ is an equality on $T$-points for every geometrically reduced $k$-scheme $T$.
\item[(ii)] The functor $\calHom(H, G)^+$ is represented by a geometrically reduced $k$-scheme locally of finite type.
\end{itemize}
Furthermore,
\begin{itemize}
\item[(1)] if $G$ is commutative, then $\calHom(H, G)^+$ is a smooth commutative $k$-group scheme whose identity component is wound unipotent and whose \'etale component group has finitely-generated group of $k_s$-points.
\item[(2)] if $S$ is a semiabelian variety, then $\calHom(H, S)^+$ is \'etale with finitely-generated group of $k_s$-points.
\item[(3)] if $U$ is wound unipotent, then $\calHom(H, U)^+$ is of finite type. If $U$ is commutative wound unipotent, then $\calHom(H, U)^+$ is finite type, smooth, and unipotent with wound identity component.
\end{itemize}
\end{theorem}

\begin{proof}
Consider the map $\calMor(H, G)^+ \rightarrow \calMor(H \times H, G)^+$ sending a map $f$ to the map $(h, h') \mapsto f(hh')f(h')^{-1}f(h)^{-1}$. (This map is defined by applying Yoneda's Lemma to the category of geometrically reduced $k$-schemes.) The preimage of $c_1 \in \calMor(H \times H, G)^+(k)$, the constant map to $1$, is a scheme $X$ such that the functor $h_X$ of morphisms into $X$ is isomorphic to $\calMor(H, G)^+$. As usual, we may replace $X$ by its maximal geometrically reduced subscheme to obtain a locally finite type, geometrically reduced $k$-scheme representing $\calMor(H, G)^+$. This yields a geometrically reduced $k$-scheme locally of finite type $\calHom(H, G)^+$ which defines a subfunctor of $\calHom(H, G)$ and satisfies (i). When $G$ is commutative, the functor $\calHom(H, G)^+$ defines a commutative group valued functor on the category of geometrically reduced $k$-schemes, hence is even a commutative $k$-group scheme, hence smooth. The finite generation of the $k_s$-points of the component group follows from the corresponding property of $\calMor(H, G)^+$ (Theorem \ref{maintheoremtake2}). Property (3) follows from the corresponding property for the scheme $\calMor(H, U)^+$ (Theorem \ref{Mor(X,U)^+}). It remains to prove that $\calHom(H, S)^+$ is \'etale when $S$ is semiabelian, and that $\calHom(H, G)^+$ has wound identity component in general (for commutative $G$).

First consider the case in which $H$ is connected. The functor $\calHom(H, G)^+$ is also represented by the maximal smooth $k$-subgroup scheme of the kernel of the map $$\calMor((H, 1), (G, 1))^+ \rightarrow \calMor(H \times H, G)^+$$ sending a map $f$ to the map $(h, h') \mapsto f(hh') - f(h) - f(h')$. In particular, the properties (1) and (2) follow from the corresponding properties of the group scheme $\calMor((H, 1), (G, 1))^+$ listed in Theorem \ref{modulipointedmorphisms} and Theorem \ref{pointedmorsemiabelianvar}.

Now suppose that $H = E$ is \'etale. By Galois descent, we may assume that $k = k_s$, hence that $E$ is constant. We may replace $E$ by $E/\mathscr{D}E$ and thereby assume that $E$ is commutative, and we may further assume that $E = \Z/N\Z$ for some $N \geq 0$. For (2), we note that $\calHom(E, S)^+ = \calHom(E, S[N])^+$, and $S[N]$ is finite because $S$ is semiabelian. Then $\calHom(E, S[N])^+$ is represented by the constant $k$-group scheme associated to the finite group $\Hom(E(k), S[N](k))$, which proves (2) when $H$ is \'etale.

Next we show that $\calHom(E, G)^+$ has wound unipotent identity component for commutative $G$. Once again, we have $\calHom(E, G)^+ = \calHom(E, G[N])^+$. We may replace $G[N]$ by its maximal smooth $k$-subgroup scheme (because this does not affect the restriction of $\calHom(E, G[N])$ to the category of geometrically reduced $k$-schemes, hence does not affect $\calHom(E, G[N])^+$) and thereby assume that $G$ is smooth and $N$-torsion. The exact sequence
\[
1 \longrightarrow \calHom(E, G^0)^+ \xlongrightarrow{\psi} \calHom(E, G)^+ \longrightarrow \calHom(E, G/G^0)^+,
\]
of sheaves on the geometrically reduced \'etale site of $k$ is an exact sequence of $k$-group schemes, by Lemma \ref{sesgeometalesite}. Therefore, in order to prove that $\calHom(E, G)^+$ has wound unipotent identity component, it suffices to show the same for $\calHom(E, G^0)^+$, so we may assume (for the purpose of showing that $\calHom(E, G)^+$ has wound unipotent identity component) that $G$ is smooth, connected, and $N$-torsion. It follows that $G$ is (wound) unipotent. Then, since $E = \Z/N\Z$, we see that $\calHom(E, G)^+ \simeq G$ is wound unipotent.

Finally, we show in general that $\calHom(H, G)^+$ has wound unipotent identity component, and that $\calHom(H, S)^+$ is \'etale. By Lemma \ref{sesgeometalesite} and the already-treated case when $H$ is connected, the sequence
\[
1 \longrightarrow \calHom(H/H^0, S)^+ \longrightarrow \calHom(H, S)^+ \xlongrightarrow{\beta} \calHom(H^0, S)^+
\]
of $k$-group schemes is exact. Using the already-treated cases when $H$ is connected or \'etale, it follows that $\calHom(H, S)^+$ is \'etale, which completes the proof of (2). It remains to prove that $\calHom(H, G)^+$ has wound identity component for general $G$. Choosing an exact sequence as in Proposition \ref{extsemiabalmunip}, we obtain by Lemma \ref{sesgeometalesite} again an exact sequence of $k$-group schemes
\[
1 \longrightarrow \calHom(H, G') \longrightarrow \calHom(H, G) \longrightarrow \calHom(H, S).
\]
We find therefore that $\calHom(H, G')$ and $\calHom(H, G)$ have the same identity component, so we are reduced  to the case when $G$ is almost unipotent. We are free to replace $G$ by its maximal smooth $k$-subgroup scheme, and thereby assume that $G$ is smooth unipotent. In the exact sequence of sheaves on the geometrically reduced \'etale site
\[
1 \longrightarrow \calHom(H, G^0)^+ \longrightarrow \calHom(H, G)^+ \longrightarrow \calHom(H, G/G^0)^+,
\]
the term on the right agrees with $\calHom(H/H^0, G/G^0)^+$, hence is \'etale. Applying Lemma \ref{sesgeometalesite} yet again, therefore, we may assume that $G$ is connected -- that is,  $G$ is wound unipotent, and in this case we are done by (3).
\end{proof}

\begin{example}
\label{endof1dimlgpsexample}
By a theorem of Russell \cite[Th.\,3.1]{russell}, any $1$-dimensional wound unipotent group (equivalently, any nontrivial $k$-form of $\Ga$) $U$ over a field $k$ has only finitely many $k_s$-endomorphisms. It follows that $\calEnd(U)^+ := \calHom(U, U)^+$ is finite \'etale.
\end{example}

To show that this entire theory is not vacuous, we will also give an example of two wound $k$-forms $V$ and $U$ of $\Ga$ such that $\calHom(V, U)^+$ is positive-dimensional. Before doing so, however, we require a couple of auxiliary results which will allow us to give an explicit description of all homomorphisms from a unipotent group into $\Ga$ over an arbitrary base.

%%% Lemma: Division by p-polynomials
\begin{lemma}$($Division by $p$-polynomials$)$
\label{divisionbyp-polynomials}
Let $A$ be a ring of characteristic $p$, and let $F \in A[X_1, \dots, X_n]$ be a $p$-polynomial. Suppose that the leading term in $X_{i_0}$ of $F$ is $uX_{i_0}^{p^{n_{i_0}}}$ for some $n_{i_0} \geq 0$ and some $u \in A^{\times}$ $($so in particular $X_{i_0}$ appears in $F$$)$. Then
\begin{itemize}
\item[(i)] for every polynomial $H \in A[X_1, \dots, X_n]$, there is a unique polynomial $H' \in A[X_1, \dots, X_n]$ such that $H \equiv H' \pmod{F}$ and ${\rm{deg}}_{X_{i_0}}(H') < p^{n_{i_0}}$;
\item[(ii)] if $H$ is a $p$-polynomial, then so is $H'$.
\end{itemize}
\end{lemma}

\begin{proof}
The existence and uniqueness of $H'$ given in (i) are just the usual division algorithm in the ring $R[X_{i_0}]$, with $R := A[X_i \mid i \neq i_0]$, using the fact that the leading coefficient in $X_{i_0}$ is a unit of $R$. For the statement in (ii), let $G \in A[X_1, \dots, X_n]$ be a $p$-polynomial of minimal degree in $X_{i_0}$ such that $H \equiv G \pmod{F}$. We will show that $G = H'$ by showing that ${\rm{deg}}_{X_{i_0}}(G) < p^{n_{i_0}}$. 

Indeed, suppose to the contrary that ${\rm{deg}}_{X_{i_0}}(G) \geq p^{n_{i_0}}$. Then we may write $G = \sum_{i=1}^n \sum_{j=0}^{N_i} c_{i,\,j}X_i^{p^j}$ for some $N_i \geq -1$ with $N_{i_0} \geq n_{i_0}$ and $c_{i_0,\,N_{i_0}} \neq 0$. Write $F = \sum_{i=1}^n \sum_{j=0}^{n_i} b_{i,\,j}X_i^{p^j}$ with $b_{i_0,\, n_{i_0}} = u$. Then $$G - c_{i_0,\,N_{i_0}}(u^{-1}F)^{p^{N_{i_0} - n_{i_0}}}$$ is a $p$-polynomial congruent to $G$ modulo $F$ whose degree in $X_{i_0}$ is strictly smaller than that of $G$. This violates the minimality of ${\rm{deg}}_{X_{i_0}}(G)$ and completes the proof of the lemma.
\end{proof}

We now describe all homomorphisms into $\Ga$ from a group given as the vanishing locus of a $p$-polynomial.

%%% Proposition: Maps from vanishing loci of p-polynomials into G_a
\begin{proposition}
\label{mapsfromvanishingppolynintoGa}
Let $k$ be a field of characteristic $p > 0$, let $F \in k[X_1, \dots, X_n]$ be a $p$-polynomial, and suppose that ${\rm{deg}}_{X_{i_0}}(F) = p^{N_{i_0}}$ for some $N_{i_0} \geq 0$ and some $1 \leq i_0 \leq n$. Let $U := \{F = 0\} \subset \mathbf{G}_{a,\,k}^n$ be the $k$-group scheme defined as the vanishing locus of $F$. Then for any $k$-scheme $T$, every $T$-group homomorphism $U_T \rightarrow \mathbf{G}_{a,\,T}$ is of the form $G(X_1, \dots, X_n)$ for a unique $p$-polynomial $G \in {\rm{H}}^0(T, \calO_T)[X_1, \dots, X_n]$ with ${\rm{deg}}_{X_{i_0}}(G) < p^{N_{i_0}}$.
\end{proposition}

\begin{proof}
Thanks to the uniqueness, the assertion is Zariski local on $T$, so we may assume that $T = {\rm{Spec}}(A)$ is affine. The uniqueness of the $p$-polynomial $G$ associated to a given homomorphism follows from the uniqueness assertion in Lemma \ref{divisionbyp-polynomials}. It remains to prove existence. Since an $A$-group homomorphism $\mathbf{G}_{a,\,A}^n \rightarrow \mathbf{G}_{a,\,A}$ is given by a $p$-polynomial, and applying the division lemma \ref{divisionbyp-polynomials}, we see that the assertion we must prove is that any $A$-homomorphism $U_A \rightarrow \mathbf{G}_{a,\,A}$ extends to a homomorphism $\mathbf{G}_{a\,A}^n \rightarrow \mathbf{G}_{a,\,A}$ (via the inclusion $U \hookrightarrow \Ga^n$ coming from the definition of $U$ as the vanishing locus of $F$).

We first note that the assertion is true after passing to $T_{k_{\perf}}$. Indeed, by \cite[Cor.\,B.2.7]{cgp}, 
$U_{k_{\rm{perf}}}$ is a vector group, so by \cite[Cor.\,B.1.12]{cgp}, $U_{k_{\perf}} \subset (\mathbf{G}_{a,\,k_{\perf}})^n$ is a $k_{\perf}$-group direct factor, hence every homomorphism from $U_{A_{k_{\perf}}}$ into $\mathbf{G}_{a,\,A_{k_{\perf}}}$ extends to 
one from $(\mathbf{G}_{a,\,A_{k_{\perf}}})^n$. We therefore obtain a $p$-polynomial $G \in A_{k_{\perf}}[X_1, \dots, X_n]$ defining our homomorphism $\phi$. On the other hand, because $\phi$ is defined over $A \subset A_{k_{\perf}}$, there is a polynomial (not necessarily a $p$-polynomial) $H \in A[X_1, \dots, X_n]$ such that $G \equiv H \pmod{F}$. Applying the division Lemma \ref{divisionbyp-polynomials}, we obtain remainders $G' \in A_{k_{\perf}}[X_1, \dots, X_n]$ and $H' \in A[X_1, \dots, X_n]$ both of degree $< p^{N_{i_0}}$ in $X_{i_0}$ such that $G'$ is a $p$-polynomial and $G \equiv G' \pmod{F}$ and $H \equiv H' \pmod{F}$. It follows that $G'$ and $H'$ are both remainders for $G$ modulo $F$ over $A_{k_{perf}}[X_1, \dots, X_n]$ with respect to $X_{i_0}$, hence -- by the uniqueness assertion of Lemma \ref{divisionbyp-polynomials} -- $G' = H'$. It follows that $G' \in A[X_1, \dots, X_n]$ is a $p$-polynomial {\em over $A$} which defines $\phi$. This completes the proof of the proposition.
\end{proof}

\begin{example}
\label{Homschemeexample}
We now compute an example of a positive-dimensional Hom scheme between $1$-dimensional wound unipotent groups. Let $k$ be an imperfect field of characteristic $p$, let $a \in k - k^p$, and consider the groups
\[
V := \{X^{p^2} - X + aY^{p^2} = 0\} \subset \Ga^2,
\]
\[
U := \{X^p - X + aY^p = 0\} \subset \Ga^2.
\]
We first note that $V$ and $U$ are smooth because they are defined by $p$-polynomials with nonzero linear part. Next, the map $V \rightarrow \Ga$, $(X, Y) \mapsto X$, is surjective with connected kernel, hence $V$ is connected. Similar considerations with the map $U \rightarrow \Ga$, $(X, Y) \mapsto X$, show that $U$ is connected. Finally, $U$ and $V$ are wound because the $p$-polynomials defining them have principal parts with no nontrivial zeroes over $k$ \cite[Lem.\,B.1.7]{cgp}. We will compute $\calHom(V, U)^+$.

Consider the group
\[
W := \{X^{p^2} - X + aY^p = 0\} \subset \Ga^2.
\]
Similar arguments to those above show that $W$ is wound unipotent (though all we will require for the calculation to follow is that $W$ is smooth). We have a bi-additive map
\[
b\colon W \times V \rightarrow U
\]
given by the formula
\[
b((X', Y'), (X, Y)) := (X(X')^p + X^pX', XY' + X'Y^p).
\]
This map defines a $k$-group homomorphism $\phi_b\colon W \rightarrow \calHom(V, U)^+$. We will show that, for $p > 2$, $\phi_b$ is an isomorphism. This amounts to showing that $\phi_b$ induces an isomorphism on $T$-points for every geometrically reduced, finite type $k$-scheme (or even just every smooth $k$-scheme) $T$.

So suppose given such a $T$. We first check that $\phi_b$ is injective on $T$-points. By Proposition \ref{mapsfromvanishingppolynintoGa}, every $T$-homomorphism $V_T \rightarrow \mathbf{G}_{a,\,T}$ is given by a {\em unique} $p$-polynomial $G \in {\rm{H}}^0(T, \calO_T)[X, Y]$ such that ${\rm{deg}}_X(G) \leq p$. The uniqueness implies that an element $(X', Y') \in W(T)$ defines the zero homomorphism $V_T \rightarrow U_T$ via $b$ if and only if $X' = Y' = 0$.

Now we check surjectivity. Suppose given a $T$-homomorphism $\psi\colon V_T \rightarrow U_T$. Applying Proposition \ref{mapsfromvanishingppolynintoGa} to each coordinate of the composition $V_T \xrightarrow{\psi} U_T \hookrightarrow \mathbf{G}_{a,\,T}^2$, we see that $\psi$ is of the form
\[
(X, Y) \mapsto (cX + dX^p + F(Y), eX + fX^p + G(Y))
\]
for some $c, d, e, f \in {\rm{H}}^0(T, \calO_T)$ and some $p$-polynomials $F, G \in {\rm{H}}^0(T, \calO_T)[Y]$. The fact that this lands in $U_T$ says that
\[
(cX + dX^p + F(Y))^p - (cX + dX^p + F(Y)) + a(eX + fX^p + G(Y))^p = 0,
\]
or eliminating terms containing $X^{p^2}$ by using the equation for $(X, Y) \in V$,
\[
(d^p - c + af^p)X + (c^p - d + ae^p)X^p + (F(Y)^p - F(Y) + aG(Y)^p - (ad^p +a^2f^p)Y^{p^2}) = 0.
\]
Using the uniqueness aspect of Proposition \ref{mapsfromvanishingppolynintoGa} for maps $V_T \rightarrow \mathbf{G}_{a,\,T}$, we conclude from the above equation that
\begin{equation}
\label{Homschemeexampleeqn1}
c = d^p + af^p
\end{equation}
\begin{equation}
\label{Homschemeexampleeqn2}
d = c^p + ae^p
\end{equation}
\begin{equation}
\label{Homschemeexampleeqn3}
F(Y)^p - F(Y) + aG(Y)^p - (ad^p + a^2f^p)Y^{p^2} = 0.
\end{equation}
If $F(Y) \neq 0$, then let $g$ be the leading coefficient of $F$. Then equation (\ref{Homschemeexampleeqn3}) implies that
\[
g^p + ar_1^p + a^2r_2^p = 0
\]
for some $r_1, r_2 \in {\rm{H}}^0(T, \calO_T)$. Assume from now on that $p > 2$. Then it follows that, for every $t \in T(k_s)$, one has $g(t) = 0$. Because $T$ is geometrically reduced of finite type over $k$, $T(k_s)$ is Zariski dense in $T$, so we conclude that $g = 0$, a contradiction. It follows that in fact $F(Y) = 0$. Using equation (\ref{Homschemeexampleeqn3}) again, we deduce that 
\begin{equation}
\label{Homschemeexampleeqn4}
a(fY^p)^p = (G(Y) - dY^p)^p.
\end{equation}
We conclude that $f(t) = 0$ for every $t \in T(k_s)$, hence that $f = 0$. Therefore, because $T$ is reduced, equation (\ref{Homschemeexampleeqn4}) implies that
\[
G(Y) = dY^p.
\]
On the other hand, because $f = 0$, equation (\ref{Homschemeexampleeqn1}) implies that $c = d^p$, and substituting this into (\ref{Homschemeexampleeqn2}), we find that $d = d^{p^2}  + ae^p$. That is, $(d, e) \in W(T)$. We therefore find that $\psi$ is given by the formula
\[
(X, Y) \mapsto (d^pX + dX^p, eX + dY^p) = b((d, e), (X, Y)).
\]
We therefore see that $\psi = \phi_b((d, e)) \in \phi_b(W(T))$. It follows that $\phi_b$ is surjective on $T$-points. We therefore conclude that $\phi_b$ is an isomorphism when $p > 2$, as claimed.

When $p = 2$, one can also compute $\calHom(V, U)^+$. Let
\[
W_2 := \{X^4 + X + aY^2 + a^2Z^8 = 0\} \subset \Ga^3.
\]
The map $W_2 \rightarrow \Ga^2$, $(X, Y, Z) \mapsto (X, Z)$, is surjective with connected kernel, hence $W_2$ is connected. It is also smooth, and it is wound because its defining $p$-polynomial has principal part with no nontrivial zeroes over $k$ \cite[Lem.\,B.1.7]{cgp}. We have a bi-additive map
\[
b_2\colon W_2 \times V \rightarrow U
\]
given by the formula
\[
b_2((X', Y', Z'), (X, Y)) := (X(X')^2 + aX(Z')^4 + X^2X' + aY^2(Z')^2, XY' + X^2(Z')^2 + YZ' + Y^2X').
\]
A calculation similar to the one above for $p \neq 2$, but slightly more involved, shows that $b_2$ induces a $k$-group isomorphism $W_2 \xrightarrow{\sim} \calHom(V, U)^+$. We leave the details to the interested reader.
\end{example}

Recall Theorem \ref{albanese}, which asserts that generalized Jacobians of curves have an Albanese type property with respect to maps into commutative group schemes. We would like to show that, for curves and (certain) groups starting life over a field, this remains true after base change by a geometrically reduced $k$-scheme. We first require a lemma.

%%% Lemma: Solvable groups contain G_a iff they do over a separable extension
\begin{lemma}
\label{containGasepble}
Let $K/k$ be a separable extension of fields, and let $G$ be a solvable, finite type $k$-group scheme. Then $G$ admits a $k$-group inclusion from $\Ga$ if and only if $G_K$ admits a $K$-group inclusion from $\mathbf{G}_{a,\,K}$.
\end{lemma}

\begin{proof}
The only if direction is clear, so we concentrate on the if direction. First consider the case in which $G$ is smooth, connected, and affine. Assume that $G$ admits no copy of $\Ga$ over $k$. Let $U \trianglelefteq G$ be the $k$-unipotent radical of $G$ -- that is, the maximal smooth connected normal unipotent $k$-subgroup of $G$. Then $G/U$ is solvable pseudo-reductive, hence commutative \cite[Prop.\,1.2.3]{cgp}, and $(G/U)_K$ is still pseudo-reductive \cite[Prop.\,1.1.9(1)]{cgp}, hence admits no nonzero $K$-homomorphism from $\Ga$. Since $U \subset G$, $U$ must be wound, hence $U_K$ is still wound. Since $G$ is an extension of $G/U$ by $U$, it follows that $G$ admits no nonzero homomorphisms from $\Ga$ over $K$.

Now consider the general case. Because $G$ admits a copy of $\Ga$ if and only if its maximal smooth $k$-subgroup scheme does, and because the formation of the maximal smooth subgroup commutes with separable extension of the ground field \cite[Lem.\,C.4.1]{cgp}, in order to prove the lemma for $G$, it is equivalent to prove it for its maximal smooth $k$-subgroup scheme. Similarly, we may replace a given $G$ by its identity component. In particular, we may assume that $G$ is smooth and connected. Proposition \ref{extsemiabalmunip} then furnishes an exact sequence
\[
1 \longrightarrow G' \longrightarrow G \longrightarrow A \longrightarrow 1
\]
with $G'$ affine and $A$ an abelian variety. Then $G$ admits a copy of $\Ga$ if and only if $G'$ does, and similarly over $K$. We may therefore replace $G$ with $G'$, and, as discussed above, we may replace $G'$ with the identity component of its maximal smooth $k$-subgroup, which is affine. Thus we have reduced to the already-treated smooth connected affine case.
\end{proof}

We now prove that the Albanese property for generalized Jacobians of curves asserted in Theorem \ref{albanese} remains true after geometrically reduced base change.

%%% Theorem: Moduli of maps from curves agree with moduli of homomorphisms from generalized Jacobians
\begin{theorem}
\label{modulifromcurvesequalsfromjacobians}
Let $X$ be a smooth curve over a field $k$, with regular compactification $\overline{X}$, and let $x \in X(k)$. Let $D \subset \overline{X}$ be a divisor with support $\overline{X}\backslash X$. Then for any finite type commutative $k$-group scheme $G$ not containing a $k$-subgroup scheme $k$-isomorphic to $\Ga$, the natural $k$-group homomorphism
\[
i_x^*\colon \calHom(\Jac_{D}(\overline{X}), G)^+ \rightarrow \calMor((X, x), (G, 1))^+
\]
is a $k$-group isomorphism.
\end{theorem}

\begin{proof}
First of all, we claim that the map $i_x$ generates $\Jac_{D}(\overline{X})$ as a $k$-group scheme. We could show this directly, but it is easier just to use Theorem \ref{albanese}, and in particular the uniqueness of the map $\psi$ in that proposition. Indeed, let $H \subset \Jac_{D}(\overline{X})$ be the $k$-subgroup generated by $i_x$. If $H \neq \Jac_{D}(\overline{X})$, then the zero map $X \rightarrow \Jac_{D}(\overline{X})/H$ admits two distinct factorizations through homomorphisms from $\Jac_{D}(\overline{X})$ -- namely, the zero map and the projection map $\Jac_{D}(\overline{X}) \rightarrow \Jac_{D}(\overline{X})/H$. This shows that in fact $i_x$ generates $\Jac_{D}(\overline{X})$, as claimed. It follows that the map of functors $\calHom(\Jac_{D}(\overline{X}), G) \rightarrow \calMor((X, x), (G, 1))$ induced by $i_x$ is an inclusion, hence the same holds for the induced map $i_x^*$ of subfunctors.

It only remains to show that $i_x^*$ is surjective as a map of $k$-groups. Because the group $\calMor((X, x), (G, 1))^+$ is smooth, it suffices to show that $i_x^*$ is surjective on $k_s$-points. That is, we must show that a morphism $f\colon (X_{k_s}, x) \rightarrow (G_{k_s}, 0)$ of pointed $k_s$-schemes extends to a $k_s$-group homomorphism $\Jac_D(\overline{X})_{k_s} \rightarrow G_{k_s}$. By Theorem \ref{albanese}, this is true (for the map $f$) if we replace $D$ by some possibly larger divisor $D'$ with the same support. By Lemma \ref{thickendivisor}, the natural map $\Jac_{D'}(\overline{X}) \rightarrow \Jac_D(\overline{X})$ has split unipotent kernel. Because $G$ admits no $k$-subgroup $k$-isomorphic to $\Ga$, the same holds over $k_s$ by Lemma \ref{containGasepble}. Therefore, the $k_s$-homomorphism $\Jac_{D'}(\overline{X})_{k_s} \rightarrow G_{k_s}$ factors through a homomorphism $\Jac_D(\overline{X})_{k_s} \rightarrow G_{k_s}$.
\end{proof}

%%% Corollary: Geometric class field theory over a geometrically reduced base
\begin{corollary}
\label{albanesebase}
Let $X$ be a smooth curve over a field $k$, with regular compactification $\overline{X}$, and let $x \in X(k)$. Let $D \subset \overline{X}$ be a divisor with support $\overline{X}\backslash X$. Let $G$ be a finite type commutative $k$-group scheme not containing a $k$-subgroup scheme $k$-isomorphic to $\Ga$, and let $T$ be a geometrically reduced $k$-scheme. Then for any pointed $T$-morphism $f \colon (X_T, x_T) \rightarrow (G_T, 0_T)$, there is a unique $T$-group scheme homomorphism $\psi\colon \Jac_D(\overline{X})_T \rightarrow G_T$ such that the following diagram commutes:
\[
\begin{tikzcd}
X_T \arrow{r}{i_x} \arrow{dr}[swap]{f} & \Jac_{D}(\overline{X})_T \arrow[d, dashrightarrow, "\exists! \psi"] \\
& G_T
\end{tikzcd}
\]
\end{corollary}

%%% Appendix
\appendix

%%% Section: Non-solvable groups contain G_a
\section{Non-solvable groups contain $\Ga$}

The purpose of this section is to prove that smooth, connected, non-solvable groups over separably closed fields contain a copy of $\Ga$. While we never use this result in the present paper, it explains the restriction to solvable groups in the main theorem \ref{maintheorem}.

We require a lemma.

%%% Lemma: Torus actions on split unipotent groups produce automorphisms
\begin{lemma}
\label{torusautomoprhisms}
Let $k$ be a field, $T$ a $k$-torus, and $U$ a smooth connected unipotent $k$-group equipped with a $T$-action such that ${\rm{Lie}}(U)^T = 0$. Then for all $t \in T(k)$ outside of some closed subscheme $Z \subsetneq T$, the map $\phi_t\colon U \rightarrow U$ defined by the formula $u \mapsto (t\cdot u)u^{-1}$ is a $k$-scheme automorphism of $U$.
\end{lemma}

\begin{proof}
We may extend scalars to $k_s$ and thereby assume that $T$ is split. In particular, the action of $T$ on ${\rm{Lie}}(U)$ breaks up as a direct sum of characters. We first treat the case in which $U$ is commutative. Let $Z\subset T$ be the union of the kernels of the finitely many characters appearing in the $T$-representation ${\rm{Lie}}(U)$. Because these characters are all nontrivial by hypothesis, $Z \neq T$. We claim that $\phi_t$ is an automorphism of $U$ for $t \in (T\backslash Z)(k)$.

Indeed, since $U$ is smooth and connected, and $\phi_t$ is a $k$-homomorphism (because $U$ is commutative), it is equivalent to show that $U^t = \ker(\phi_t)$ is trivial. Let $S \subset T$ be the Zariski closure of the subgroup generated by $t$. Then $U^S = U^t$, and by \cite[Props.\,2.1.12(3), A.8.10(2)]{cgp}, $U^S$ is smooth and connected. Furthermore, ${\rm{Lie}}(U^S) = {\rm{Lie}}(U)^S = 0$, the last equality because the characters of $T$ appearing in ${\rm{Lie}}(U)$ are all nontrivial when applied to $t \in S$. It follows that $U^S = 0$. This completes the proof for commutative $U$.

Now we treat the general case by induction on ${\rm{dim}}(U)$. The case ${\rm{dim}}(U) = 0$ is trivial, so assume that $U$ is nontrivial. Because $U$ is nilpotent, for some $n \geq 0$ the $n$th central subgroup $U' := \mathscr{D}_n(U)$ is a nontrivial smooth, connected, central, {\em characteristic} $k$-subgroup of $U$. Because it is characteristic, it is invariant under the $T$-action on $U$. Let $U'' := U/U'$. Our assumption that ${\rm{Lie}}(U)$ has no nonzero $T$-invariants is inherited by $U'$ and $U''$. Therefore, by induction and the already-treated commutative case, there are closed subschemes $Z', Z'' \subsetneq T$ such that $\phi_t$ induces $k$-scheme automorphisms of $U'$, respectively $U''$, for $t \notin Z'$, respectively $Z''$. It follows formally that $\phi_t$ is a scheme-theoretic automorphism of $U$ for $t \notin Z := Z' \cup Z''$.
\end{proof}

%%% Proposition: Smooth connected non-solvable group over k_s contains G_a
\begin{proposition}
\label{nonsolvablecontainsGa}
A smooth connected non-solvable group $G$ over a separably closed field $k$ admits an injective $k$-group homomorphism from $\Ga$.
\end{proposition}

\begin{proof}
Let us first recall the open cell decomposition associated to a split maximal torus. Let $H$ be a smooth connected affine $k$-group scheme, $T \subset H$ a (split) torus, and choose a cocharacter $\lambda\colon \Gm \rightarrow T$ such that $Z_H(\lambda) = Z_H(T)$. (All cocharacters lying in the complement of a finite set of hyperplanes in the cocharacter lattice  of $T$ have this property.) Then we have the open cell decomposition associated to $\lambda$: there are split unipotent $k$-subgroups $U^+, U^- \subset H$ that are normalized by $T$, and such that the multiplication map
\[
U^- \times Z_H(T) \times U^+ \rightarrow H
\]
is an open immersion \cite[Lemma 2.1.5, Props.\,2.1.10, 2.1.8(3)]{cgp}.

Now consider the case in which $k$ is perfect. Chevalley's Theorem furnishes an exact sequence
\[
1 \longrightarrow L \longrightarrow G \longrightarrow A \longrightarrow 1
\]
with $L$ smooth, connected, and affine, and $A$ an abelian variety. Over perfect fields, therefore, the proposition reduces to the affine case. So assume that $G$ is affine, and choose a maximal torus $T \subset G$. Since $T$ is maximal, $Z_G(T)/T$ admits no nontrivial torus, and is therefore unipotent \cite[Ch.\,IV, Cor.\,11.5(2)]{borelalggroups}. It follows that $Z_G(T)$ is solvable. Since $G$ is not solvable, we therefore have that $G \neq Z_G(T)$. Therefore, in the open cell decomposition associated to a ``generic'' cocharacter of $T$, the groups $U^+$ and $U^-$ cannot both be trivial. So $G$ contains a nontrivial split unipotent $k$-subgroup, hence contains a copy of $\Ga$. This completes the proof when $k$ is perfect.

Now assume that ${\rm{char}}(k) = p > 0$ (which in particular includes the imperfect case). Then $G$ may be written as a central extension
\begin{equation}
\label{nonsolvablecontainsGapfeqn1}
1 \longrightarrow S \longrightarrow G \xlongrightarrow{\pi} L \longrightarrow 1
\end{equation}
with $S$ a semiabelian variety and $L$ smooth, connected, and affine \cite[Th.\,A.3.9]{cgp}. The non-solvability of $G$ implies the same for $L$. The left multiplication action of $S$ on $G$ makes $G$ into an $S$-torsor over $L$. Further, it is a {\em multiplicative} $S$-torsor: it lies in the kernel of the map $$m^* - \pi_1^* - \pi_2^*\colon {\rm{H}}^1(L, S) \rightarrow {\rm{H}}^1(L \times L, S),$$ where $m, \pi_i\colon L \times L \rightarrow L$ are the multiplication and projection maps. 

Let $T \subset L$ be a maximal torus, and let $U^+$ and $U^-$ denote the split unipotent $k$-subgroups of $L$ appearing in the open cell decomposition associated to a ``generic'' cocharacter of $T$, as discussed above. Because $L$ is not solvable, at least one of these two groups is nontrivial. We will show that, for $U = U^+$ or $U^-$, the sequence (\ref{nonsolvablecontainsGapfeqn1}) splits when restricted to $\pi^{-1}(U) \subset G$. Consider for $t \in T(k)$ the sequence of maps
\[
U \xrightarrow{\gamma} T \times U \xrightarrow{\psi} L,
\]
where the first map is $u \mapsto (t, u)$ and the second is $(t, u) \mapsto tut^{-1}u^{-1}$. Because $U$ is normalized by $T$, this second map lands inside $U$. We claim that, for suitable $t \in T(k)$, the composed map $\phi_t\colon U \rightarrow U$, $u \mapsto tut^{-1}u^{-1}$, is a $k$-scheme automorphism. Indeed, we have ${\rm{Lie}}(Z_L(T)) = {\rm{Lie}}(L)^T$, where the invariants are with respect to the adjoint action of $T$. It follows that ${\rm{Lie}}(U)^T = 0$, so the claim follows from Lemma \ref{torusautomoprhisms}.

The multiplicativity of the $S$-torsor $\alpha$ given by $G \rightarrow L$ implies that the pullback of $\alpha$ along the product of maps is the sum of the pullbacks: given a scheme $X$ and maps $f, g \colon X \rightarrow L$, one has 
\[
(fg)^*(\alpha) = f^*(\alpha) + g^*(\alpha).
\] 
Indeed, this follows by factoring the map $fg\colon X \rightarrow L$ as the composition
\[
X \xrightarrow{f \times g} L \times L \xrightarrow{m} L,
\]
and then using the multiplicativity of $\alpha$ to conclude that
\begin{align*}
(fg)^*(\alpha) = (m \circ (f \times g))^*\alpha = (f \times g)^*(m^*(\alpha)) = (f \times g)^*(\pi_1^*(\alpha)) + (f \times g)^*((\pi_2)^*(\alpha)) \\
= (\pi_1 \circ (f \times g))^*(\alpha) + (\pi_2 \circ (f \times g))^*(\alpha) = f^*(\alpha) + g^*(\alpha).
\end{align*}
An easy induction shows that the same holds for any finite product of maps into $L$. Applying this when $X = L$ and $f = g$ is the constant map to $1_L$, we deduce that the restriction of $\alpha$ to the identity is trivial. Applying it when $f$ is the identity map of $L$ and $g = [-1]$ the inversion map of $L$, we deduce that $[-1]^*(\alpha) = -\alpha$.

In particular, because of the definition of $\psi$ as a product of maps and their inverses, $\psi^*(\alpha) = i_T^*\alpha + i_U^*\alpha - i_T^*\alpha - i_U^*\alpha = 0$, where $i_T\colon T \rightarrow L$ is the inclusion and $i_U$ is defined similarly. Pulling back further along $\gamma$, we find that $\phi_t^*(i_U^*(\alpha)) = 0$. Because $\phi_t$ is a $k$-scheme automorphism of $U$, we deduce that $i_U^*\alpha = 0$. That is, $\pi^{-1}(U) \rightarrow U$ admits a scheme-theoretic section. In order to complete the proof of the proposition, we must show that there is a group-theoretic section.

Write $\pi^{-1}(U) = S \times U$ as $S$-torsors over $U$. Since $S \subset \pi^{-1}(U)$ is central, the group law on $\pi^{-1}(U)$ then takes the form $(s_1, u_1)\cdot(s_2, u_2) = (s_1 + s_2 + h(u_1, u_2), u_1u_2)$ for some $h\colon U \times U \rightarrow S$. We claim that $f(u, v) = f(u, 1) + f(1, v) - f(1,1)$ for any map $f \colon U \times U \rightarrow S$ into a semiabelian variety. Indeed, letting $g(u, v) := f(u, v) - f(u, 1) - f(1, v) + f(1, 1)$, we have that $g(u, 1) = 0$ and $g(1, v) = 0$. It therefore suffices (is equivalent, in fact) to check that any $g$ satisfying these conditions is the $0$ map. Writing $S$ as an extension of an abelian variety by a torus, we are reduced to checking the assertion separately in the abelian variety and torus cases. In the abelian variety case, it follows from the fact that any map from affine space to an abelian variety is constant. In the torus case, it follows from the fact that affine space has no nonconstant units.

We therefore see that $h(u, v) = h(u, 1) + h(1, v) - h(1, 1)$. So we may write $h(u, v) = f_1(u) + f_2(v) + s$ for some $s \in S(k)$, and by modifying $s$ we may assume that $f_1(1) = f_2(1) = 0$. The associative law on $\pi^{-1}(U)$ yields the identity
\[
h(u, v) + h(uv, w) = h(v, w) + h(u, vw),
\]
or
\[
f_2(v) + f_1(uv) = f_1(v) + f_2(vw).
\]
Setting $v = 1$ yields $f_1(u) = f_2(w)$. Therefore, $f_1$ and $f_2$ are constant, hence so too is $h$, say $h = s_0 \in S(k)$. Then the map $U \rightarrow S \times U \simeq \pi^{-1}(U)$ given by $u \mapsto (-s_0, u)$ is a group-theoretic section to the map $\pi^{-1}(U) \rightarrow U$, as desired.
\end{proof}

%%% Section: Explanation of hypotheses in the main theorem
\section{Explanation of hypotheses in the main theorem}
\label{explanationofhypsection}

The main result of the present paper, Theorem \ref{maintheorem}, asserts the existence of suitable moduli spaces of morphisms from schemes into algebraic groups. The theorem imposes certain assumptions on the schemes and the groups, as well as on the category on which the moduli space represents the functor of morphisms. The purpose of this section is to explain why these assumptions are the natural ones.

Let $X$ be a $k$-scheme of finite type, and let $G$ be a smooth connected $k$-group scheme. In general, it is completely hopeless to expect there to be a scheme representing (on any reasonably large category) the functor of morphisms from $X$ to $G$. For example, if we take $G = \Ga$, then, for a $k$-scheme $T$, a $T$-morphism from $X_T$ to $\mathbf{G}_{a,\,T}$ is just a global section of $X_T$. The collection of such sections for affine $X$ of positive dimension will typically be far too large to be representable; for example, if $X = \A^1$, then it is just the space of all polynomials with coefficients in $\Gamma(T, \calO_T)$. Without some sort of boundedness condition on these polynomials -- say, requiring the degree to be bounded -- this space is too large to represent by a scheme. Thus we should assume that our group $G$ does not contain a copy of $\Ga$. If, for example, $G = U$ is unipotent, then this is equivalent to requiring that $U$ be wound. 

But the problem with groups containing $\Ga$ causes a corresponding problem with the category of all $k$-schemes: it is simply too large. Indeed, while woundness is a property inherited by separable extensions of the field $k$ \cite[Prop.\,B.3.2]{cgp}, it is not generally inherited by inseparable extensions. In fact, over a perfect field every smooth connected unipotent group is split \cite[Thm.\,15.4(iii)]{borelalggroups}, so a nontrivial wound unipotent group acquires subgroups isomorphic to $\Ga$ over some finite purely inseparable extension. It is therefore hopeless, if one wishes to construct moduli spaces of morphisms into unipotent groups in any interesting generality, to construct them on the category of all $k$-schemes, or even those of finite type. Instead, we must restrict ourselves to those whose structure may be probed by separable points. A natural candidate is the category of smooth $k$-schemes. We instead deal with the somewhat larger category of geometrically reduced $k$-schemes.

In addition, we are required to impose certain restrictions on the source scheme $X$. Indeed, if $X$ is non-reduced, then there can be many morphisms from $X$ into the infinitesimal subgroups of $U$. For example, if $X = \A^1_Y$ with $Y$ non-reduced, and $U$ is a $k$-form of $\Ga$, then the kernel of the Frobenius isogeny of $U$ is isomorphic to $\alpha_p$, and a morphism from $X$ into $\alpha_p$ consists of a polynomial with coefficients lying in the ideal of $p$-nilpotents of $Y$ (i.e., those elements of $\Gamma(Y, \calO_Y)$ whose $p$th power is $0$). Again, without some bound on the degree of these polynomials there is no hope to obtain representability for so large a space. We are therefore led to assume in addition that $X$ is reduced.

Finally, the $\Ga$ problem also entails an additional restriction on the scheme $X$ beyond mere reducedness. Suppose that $K/k$ is a non-separable extension field, and $G = U$ is a wound unipotent group which acquires a $\Ga$ subgroup over $K$. Then for a $K$-scheme $Y$, a $k$-morphism into $U$ is the same as a $K$-morphism into $U\otimes_k K$. In particular, if $X$ is an integral $k$-scheme with function field $K$, then once again we run into the problem that the space of $k$-morphisms from $X$ into $U$ is usually too large to be representable (because the $\Ga$ inside $U_K$ will spread out to a $\Ga$ inside $U_Y$ for some open subscheme $Y$ of $X$). We therefore require that the function field of each irreducible component of the reduced $X$ be separable over $k$. Equivalently, we must (typically) restrict our attention to the situation in which $X$ is not merely reduced, but {\em geometrically} reduced in order to hope for representability of the space of morphisms from $X$ into $U$.

To summarize, the natural context in which it is reasonable to expect the functor sending a $k$-scheme $T$ to the group $\Mor_T(X_T, G_T)$ to be representable by a scheme is when we restrict $T$ to lie in the category of geometrically reduced $k$-schemes, $X$ to be geometrically reduced, and $G$ not to contain a copy of $\Ga$. Since we are working on the category of all geometrically reduced $k$-schemes, we in fact do not merely want $G$ to contain a copy of  $\Ga$ over $k$, but we wish for this to remain true over $k_s$. For smooth connected $G$, this implies that $G$ is solvable (Proposition \ref{nonsolvablecontainsGa}), hence the solvability assumption in Theorem \ref{maintheorem}.

\noindent \address
\vspace{.3 in}

\noindent \email

\end{document}